\documentclass[11pt]{article}

\usepackage[a4paper,margin=28mm]{geometry}
\usepackage[T1]{fontenc}
\usepackage{lmodern}
\usepackage[nopatch=footnote]{microtype}
\usepackage{amsmath,amssymb,amsthm,mathtools}
\usepackage{enumitem}
\usepackage{xcolor}
\usepackage[colorlinks=true,linkcolor=blue!55!black,citecolor=blue!55!black,urlcolor=blue!55!black]{hyperref}
\usepackage[nameinlink,noabbrev]{cleveref}
\usepackage{setspace}

\allowdisplaybreaks
\setlist[itemize]{leftmargin=2em,itemsep=0.25em,topsep=0.4em}
\setlist[enumerate]{leftmargin=2.2em,itemsep=0.25em,topsep=0.4em}

\newtheorem{theorem}{Theorem}[section]
\newtheorem{proposition}[theorem]{Proposition}
\newtheorem{lemma}[theorem]{Lemma}

\theoremstyle{definition}

\newtheorem{remark}[theorem]{Remark}

\newcommand{\R}{\mathbb R}
\newcommand{\Sph}{\mathbb S}
\newcommand{\B}{\mathbb B}
\newcommand{\cA}{\mathcal A}
\newcommand{\cE}{\mathcal E}
\newcommand{\cM}{\mathcal M}
\newcommand{\cC}{\mathcal C}
\newcommand{\cD}{\mathcal D}
\newcommand{\Sing}{\operatorname{Sing}}
\newcommand{\Tr}{\operatorname{Tr}}
\newcommand{\dist}{\operatorname{dist}}
\newcommand{\dd}{\,\mathrm d}
\newcommand{\eps}{\varepsilon}
\newcommand{\weakto}{\rightharpoonup}
\newcommand{\loc}{\mathrm{loc}}
\newcommand{\Hdim}{\dim_{\mathcal H}}
\newcommand{\1}{\mathbf 1}

\DeclareMathOperator{\degmap}{deg}

\title{Global Minimality and Rigidity of the Constraint Map Vortex}
\author{Bin Deng \quad Jiahuan Li \quad Yilu Liu \quad Xi-Nan Ma}
\date{}

\begin{document}
\maketitle

\begin{abstract}
We consider the minimization problem
\[
\min\left\{\int_{B_1}|Du|^2:\
u\in W^{1,2}(B_1;\mathbb R^n),\quad
u=x\ \text{on }\partial B_1,\quad |u|\ge a\right\},
\quad 0<a<1.
\]
Figalli, Guerra, Kim, and Shahgholian proved that the canonical radial vortex is the unique global minimizer for $n\ge7$, and asked whether the same holds in dimensions \(3\le n \le6\). We answer this question affirmatively, thereby completing the global minimality and rigidity of the constraint map vortex in every dimension \(n\ge3\).
\end{abstract}

\section{Introduction}\label{sec:introduction}

A minimizing constraint map minimizes the Dirichlet energy
\[
 E(u;\Omega)=\int_\Omega |Du|^2\dd x
\]
under the requirement that its image lie in the closure of a prescribed domain.
This variational problem combines two structures that are usually studied
separately. As in harmonic map theory, nontrivial boundary topology may force
discontinuities and tangent maps; as in obstacle problems, the constraint gives
rise to a coincidence set and an associated free boundary. The subject grew
out of the study of parametric variational problems with obstacles and harmonic
maps into manifolds with boundary
\cite{Hildebrandt1972,Tomi1972,DuzaarFuchs1986,Duzaar1987,ChenMusina1990}.
Its treatment of singularities draws on the classical regularity, compactness,
and defect theory for minimizing harmonic maps
\cite{SchoenUhlenbeck1982,SchoenUhlenbeck1984,HardtLin1987,HardtLin1989,
Luckhaus1988,HardtKinderlehrerLin1988}, while its free-boundary aspect is
related to the scalar obstacle problem \cite{Caffarelli1998}. Recent work has
established a sharper constraint map regularity theory and begun to clarify how
mapping singularities, the coincidence set, and the free boundary interact
\cite{FKSobstacle,FGKSBernoulli,FGKSsing,FGKSNotices2025,FGKSreview}.
A complementary global question is to determine which of the geometrically
possible configurations is selected by energy minimization.

The exterior of a ball with identity boundary data is the canonical degree-one
model for this question. For $n\ge3$ and $0<a<1$, set
\[
 \cA_a=
 \bigl\{u\in W^{1,2}(\B_1;\R^n):
 \Tr u=x\text{ on }\partial\B_1,\ |u|\ge a\text{ a.e. in }\B_1\bigr\},
 \qquad
 m_n(a)=\inf_{u\in\cA_a}E(u;\B_1).
\]
The same exterior-ball constrained minimization problem was recorded earlier
by Lin as an obstacle-type problem related to harmonic moduli and estimates
for the possible radius of the coincidence region
\cite[Example~1.5]{Lin2016FreeBoundary}.
The normalized boundary datum has degree one. Thus the problem retains the
degree-one topology of the classical map $x/|x|$, whose minimizing properties
were established in the classical harmonic map literature
\cite{JagerKaul1983,BCL1986,Lin1987,CoronGulliver1989}, while the exterior-ball
constraint gives rise to a contact core surrounding the resulting singularity.
There is a natural radial solution
\[
 u_{a,n}(x)=w_{a,n}(|x|)\frac{x}{|x|},
\]
whose modulus equals $a$ on a concentric ball and is given by a radial harmonic
profile outside it. Consequently, this single configuration contains the three
basic features of the theory: a topological singularity at the origin, a
contact core, and a smooth free boundary separating the core from an exterior
harmonic region.

Radial maps of this type were studied in \cite[Section~2.4]{FGKSsing}. Their
global minimality is substantially more delicate than their explicit
construction or local variational properties: the target
$\mathbb R^n\setminus\mathbb B_a$ is nonconvex, so the symmetry of the domain
and boundary data does not force a minimizer to be radial. Figalli, Guerra,
Kim, and Shahgholian proved that the radial map is the unique global minimizer
when $n\ge7$ \cite[Proposition~3.1]{FGKSreview}. Their proof uses the sharp
Hardy inequality, in the spirit of classical work on the minimality and
stability of equator maps \cite{JagerKaul1979,Baldes1984,Hong2000}. The
resulting energy comparison has the required sign precisely when
\[
 \frac{(n-2)^2}{4}>n-1,
 \qquad\text{equivalently,}\qquad n\ge7.
\]
They therefore posed the distinct global-minimality question in the remaining
dimensions $3\le n\le6$ as Problem~7.8, entitled ``Minimality of the vortex,''
in \cite{FGKSreview}.

Our main result resolves this low-dimensional problem.

\begin{theorem}\label{thm:main}
Let $3\le n\le 6$ and $0<a<1$. The map $u_{a,n}$ in
\eqref{eq:radial-map} is the unique minimizer of $E(\,\cdot\,;\B_1)$ over
$\cA_a$.
\end{theorem}

Together with \cite[Proposition~3.1]{FGKSreview}, the theorem gives the
complete classification in every dimension $n\ge3$, which is exactly the
range in which the degree-one vortex has finite Dirichlet energy.
The conclusion is genuinely global: it
excludes symmetry breaking, splitting of the boundary degree among several
singularities, and nonradial coincidence sets for minimizing configurations.

From the harmonic map side, a related body of work concerns quantitative
rigidity for sphere-valued maps: see
\cite{BernandMantelMuratovSimon2021,Topping2023} in degree one,
\cite{GuerraLamyZemas2025} in higher dimensions, and
\cite{Rupflin2023,DengSunWei2024} for general degree. Such estimates also
arise in conformal-limit problems for skyrmions and bimerons
\cite{BernandMantelMuratovSimon2021,DengIgnatLamy2025}. The present
constraint map problem has the additional feature that the modulus is free,
coupling the mapping singularity to a coincidence set and an obstacle-type
free boundary.

The proof does not replace the Hardy estimate by another direct comparison at
a fixed value of $a$. Instead, we first determine the sharp asymptotic behavior
of the minimum energy as $a\downarrow0$, and then propagate this endpoint
information to all $a\in(0,1)$. We describe three main steps.

\paragraph{The sharp endpoint asymptotic.}
Let $M_n(a)=E(u_{a,n};\B_1)$. Since $u_{a,n}$ is admissible, one has
$m_n(a)\le M_n(a)$. To obtain the reverse inequality near $a=0$, we rescale by
the obstacle radius. This gives the exact relation
\[
 m_n(a)-\sigma_{n-1}=a^n c_{a^{-1},n},
\]
where $c_{L,n}$ is the corresponding rescaled constrained minimization problem
on $\B_L$. As $L\to\infty$, this leads to the whole-space minimization problem
\[
 c_{\infty,n}=\inf\left\{
 \int_{\R^n}|D(U-x)|^2\dd x:
 U-x\in D^{1,2}(\R^n;\R^n),\quad |U|\ge1\ \text{a.e.}
 \right\}.
\]
Every minimizer of the whole-space problem has degree one at infinity and hence
has a mapping singularity. We prove that some singularity with nonzero local
degree satisfies the sharp density lower bound of the standard vortex. In
dimension three this uses the tangent-map classification of
Brezis--Coron--Lieb \cite{BCL1986}. In dimensions $4\le n\le6$, the dimension
bound for the singular set due to Lin--Wang \cite{LinWang2006} allows us to
localize the topological degree near a singular point; after passing the local
degree to a tangent map, conformal balancing, Ramanathan's theorem
\cite{Ramanathan1986}, and the Poincar\'e inequality on the sphere give the
required density estimate. An exact translated energy identity centered at
such a singularity then converts this density estimate into the sharp value of
$c_{\infty,n}$. As a consequence,
\[
 m_n(a)=\sigma_{n-1}+\kappa_n a^n+o(a^n)
 \qquad\text{as }a\downarrow0,
\]
with the same coefficient as the radial branch.

\paragraph{Propagation from the endpoint.}
For an arbitrary minimizer at parameter $a$, variations of the target radius
give a formula for $m_n'(a)$ in terms of the energy on the coincidence set.
Domain stationarity gives a Pohozaev identity, while testing the Euler equation
against the map itself relates the interior energy to the normal derivative on
$\partial\B_1$. Combining these identities with Cauchy--Schwarz produces a
differential inequality for $m_n$. Equivalently, if
\[
 \beta(a)^2=(n-1)-\frac{n-2}{\sigma_{n-1}}m_n(a),
\]
then the quantity
\[
 \Psi_n(a)=
 \frac{(1-\beta(a))(\beta(a)+n-1)^{n-1}}{a^n}
\]
is nondecreasing. The radial comparison gives
$\Psi_n(a)\le(n-1)^{n-1}$, whereas the sharp asymptotic expansion of $m_n(a)$
as $a\downarrow0$ shows that $\Psi_n(a)$ approaches the same value. Hence
equality holds for every $a\in(0,1)$, and therefore $m_n(a)=M_n(a)$ throughout
the parameter range.

\paragraph{Rigidity.}
The equality cases in the preceding inequalities first determine the full
Cauchy data of a minimizer on $\partial\B_1$. Harmonic Cauchy uniqueness and
analytic continuation then identify the map with $u_{a,n}$ throughout the
exterior noncontact annulus. On the remaining ball, we write
\[
 u=\rho v,\qquad \rho=|u|,\qquad v=\frac{u}{|u|}\in S^{n-1}.
\]
The problem then reduces to the sharp minimality of the sphere-valued vortex
with identity boundary data, due to Brezis--Coron--Lieb in dimension three and
Hong in dimensions $4\le n\le6$ \cite{BCL1986,Hong2001}. Equality forces both
the modulus and the angular map to be radial, completing the uniqueness proof.

The only dimension-dependent part of the argument is the sharp density
estimate for a singularity carrying nonzero local degree. The rescaling to the
whole-space problem, the translated energy identity, the propagation in the
parameter $a$, and the final rigidity argument are otherwise valid for every
$n\ge3$. Thus our proof provides a low-dimensional alternative to the
Hardy-based argument, relying instead on the topology of singularities and a
sharp whole-space energy estimate.

The rest of the paper is organized as follows.
\Cref{sec:preliminaries} introduces the variational framework and the radial
branch. In \cref{sec:cell} we study the rescaled whole-space minimization
problem and establish the sharp energy estimate needed as $a\downarrow0$.
\Cref{sec:global-minimality} uses this asymptotic information to prove global
minimality for every $a\in(0,1)$. Finally, \cref{sec:uniqueness} analyzes the
equality cases and proves rigidity and uniqueness.

\section{Preliminaries}\label{sec:preliminaries}

Throughout, \(n\in\{3,4,5,6\}\) unless another dimension is explicitly
indicated.  We use
\[
 \B_r(x_0)=\{x\in\R^n:|x-x_0|<r\},\qquad
 \B_r=\B_r(0),\qquad
 \sigma_{n-1}=|\Sph^{n-1}|=n|\B_1|.
\]
For matrices \(A\) and \(B\), we write
\(A:B=\operatorname{tr}(A^{\mathsf T}B)\) for their Frobenius inner product.

For \(R\in(0,1)\), let
\[
 A_n(R)=\frac{nR}{n-1+R^n}.
\]
Since
\[
 A_n'(R)=\frac{n(n-1)(1-R^n)}{(n-1+R^n)^2}>0,\qquad
 A_n(0+)=0,\qquad A_n(1)=1,
\]
there is a unique \(R=R(a)\in(0,1)\) such that
\begin{equation}\label{eq:aR}
 a=\frac{nR}{n-1+R^n}.
\end{equation}

\begin{remark}[Relation with the planar Nitsche problem]
Formally setting $n=2$ in \eqref{eq:aR} gives
\[
 \frac1a=\frac12\left(R+\frac1R\right),
\]
which is the equality case of the Nitsche bound for harmonic homeomorphisms
from $\{R<|z|<1\}$ onto $\{a<|z|<1\}$
\cite{IwaniecKovalevOnninen2011}. The corresponding constant-modulus inner
branch corresponds to the hammering region in the relaxed energy-minimization problem for
planar annuli \cite[Section~1.7]{IwaniecOnninen2012}. This does not yield a
two-dimensional case of the present full-ball problem: the map $a x/|x|$ has
infinite Dirichlet energy at the origin when $n=2$, whereas its energy is
finite for $n\ge3$.
\end{remark}

Define the radial comparison map by
\begin{equation}\label{eq:radial-map}
 u_{a,n}(x)=
 \begin{cases}
 \displaystyle a\frac{x}{|x|},&0<|x|\le R,\\[2mm]
 \displaystyle
 \frac{(n-1)|x|+R^n|x|^{1-n}}{n-1+R^n}\frac{x}{|x|},
 &R<|x|\le1.
 \end{cases}
\end{equation}
The value at \(x=0\) is irrelevant for its Sobolev class.  We shall also write
\(u_{a,n}(x)=w_R(|x|)\frac{x}{|x|}\).

Let
\[
 \overline{\cM}_a=\{y\in\R^n:|y|\ge a\}.
\]
A minimizing constraint map is a map that minimizes the Dirichlet energy
against compactly supported competitors with values in
\(\overline{\cM}_a\).  We call $x$ a regular point of a local minimizer $u$
if $u$ has a continuous representative in a neighborhood of $x$, and write
\[
 \operatorname{Reg}u=\{x:x\text{ is a regular point of }u\},
 \qquad
 \Sing u=\Omega\setminus\operatorname{Reg}u.
\]
The epsilon-regularity theorem and the resulting singular-density lower bound
\cite[Theorem~2.4 and Corollary~2.5]{FGKSsing} improve this representative to
$C^{1,1}_{\mathrm{loc}}$ on $\operatorname{Reg}u$.  Throughout the paper,
pointwise statements about $u$
on $\operatorname{Reg}u$ refer to this representative, while pointwise
statements about the coincidence set refer to the continuous representative
of the gap function constructed in \cref{prop:local-consequences}.  For such
maps, the Euler equation established in
\cite[equation~(1.2)]{FGKSsing} is
\begin{equation}\label{eq:EL}
 \Delta u=-\frac{|Du|^2}{a^2}u\,\1_{\{|u|=a\}}
 \qquad\text{in }\cD'(\Omega).
\end{equation}
There is no additional measure on the free boundary.  With
\begin{equation}\label{eq:norm-energy}
 \cE(u,x_0,r)=r^{2-n}\int_{\B_r(x_0)}|Du|^2\dd x,
\end{equation}
the monotonicity formula \cite[Lemma~2.1]{FGKSsing} gives
\begin{equation}\label{eq:monotonicity}
 \cE(u,x_0,s)-\cE(u,x_0,r)
 \ge 2\int_{\B_s(x_0)\setminus\B_r(x_0)}
 |x-x_0|^{2-n}|\partial_{r_{x_0}}u|^2\dd x
\end{equation}
whenever \(0<r<s<\dist(x_0,\partial\Omega)\).

We shall also use domain stationarity.  If \(X\in C_c^1(\Omega;\R^n)\),
\(\Phi_t(x)=x+tX(x)\), and \(u_t=u\circ\Phi_t^{-1}\), then \(u_t\) is
admissible for both signs of \(t\).  Differentiating
\[
 E(u_t;\Omega)
 =\int_\Omega
 \bigl|Du(x)D\Phi_t(x)^{-1}\bigr|^2\det D\Phi_t(x)\dd x
\]
at \(t=0\) gives
\begin{equation}\label{eq:stress}
 \int_\Omega
 \bigl(|Du|^2\delta_{ij}-2\partial_i u\cdot\partial_j u\bigr)
 \partial_i X_j\dd x=0
 \qquad
 \forall X\in C_c^1(\Omega;\R^n).
\end{equation}
\begin{proposition}\label{prop:finite-existence}
For every $0<a<1$, the infimum $m_n(a)$ is attained.
\end{proposition}

\begin{proof}
The class $\cA_a$ is nonempty: the hedgehog
\[
 h(x)=\frac{x}{|x|}
\]
belongs to $W^{1,2}(\B_1;\R^n)$ for $n\ge3$, has trace $x$ on
$\partial\B_1$, and satisfies $|h|=1\ge a$ almost everywhere. Let
$(u_j)\subset\cA_a$ be a minimizing sequence. Since
$u_j-x\in W_0^{1,2}(\B_1;\R^n)$, Poincare's inequality gives
\[
 \|u_j-x\|_{L^2(\B_1)}
 \le C\|D(u_j-x)\|_{L^2(\B_1)}
 \le C\bigl(\|Du_j\|_{L^2(\B_1)}+\sqrt n\,|\B_1|^{1/2}\bigr).
\]
Thus $(u_j)$ is bounded in $W^{1,2}(\B_1)$. Passing to a subsequence,
\[
 u_j\weakto u\quad\text{in }W^{1,2}(\B_1),
 \qquad
 u_j\to u\quad\text{in }L^2(\B_1),
\]
and $u_j(x)\to u(x)$ for almost every $x$. Continuity of the trace operator and weak closedness of the affine trace class give $\Tr u=x$ on $\partial\B_1$.

The target $\{y:|y|\ge a\}$ is nonconvex, so the obstacle constraint is not passed by weak convergence alone. The almost-everywhere convergence is the essential point: it gives
\[
 |u(x)|=\lim_{j\to\infty}|u_j(x)|\ge a
 \quad\text{for a.e. }x.
\]
Hence $u\in\cA_a$. Weak lower semicontinuity gives
\[
 E(u;\B_1)\le\liminf_{j\to\infty}E(u_j;\B_1)=m_n(a),
\]
so $u$ is a minimizer.
\end{proof}

\begin{lemma}\label{lem:modulus-truncation}
Every minimizer $u$ of $m_n(a)$ satisfies
\begin{equation}\label{eq:modulus-bounds}
 a\le |u|\le1\qquad\text{a.e. in }\B_1.
\end{equation}
\end{lemma}

\begin{proof}
The lower bound is the constraint. Since $|u|\ge a>0$, define
\[
 \rho=|u|,
 \qquad
 \omega=\frac{u}{|u|}.
\]
The map $y\mapsto \frac{y}{|y|}$ is Lipschitz on $\{|y|\ge a\}$, hence
$\omega\in W^{1,2}(\B_1;\Sph^{n-1})$. The Sobolev chain rule gives
\begin{equation}\label{eq:polar-decomp}
 Du=\omega\otimes D\rho+\rho D\omega,
 \qquad
 |Du|^2=|D\rho|^2+\rho^2|D\omega|^2,
\end{equation}
because $\omega\cdot\partial_i\omega=0$ almost everywhere.

Set
\[
 \widetilde u=\min\{\rho,1\}\,\omega.
\]
On the boundary, $\rho=|x|=1$ in the trace sense, so $\widetilde u$ has the same trace as $u$. Also $|\widetilde u|\ge a$. Thus $\widetilde u\in\cA_a$. By \eqref{eq:polar-decomp},
\begin{align*}
 E(u;\B_1)-E(\widetilde u;\B_1)
 &=\int_{\{\rho>1\}}
 \Bigl(|D\rho|^2+(\rho^2-1)|D\omega|^2\Bigr)\dd x\\
 &\ge0.
\end{align*}
Minimality of $u$ forces equality. In particular,
\[
 D(\rho-1)_+=\1_{\{\rho>1\}}D\rho=0
 \qquad\text{a.e.}
\]
The trace of $(\rho-1)_+$ is zero, so $(\rho-1)_+\in W_0^{1,2}(\B_1)$. Poincare's inequality gives $(\rho-1)_+=0$, proving $|u|\le1$ almost everywhere.
\end{proof}

\begin{proposition}\label{prop:boundary-collar}
Let $u$ minimize $m_n(a)$. There exists $\delta=\delta(u,a)>0$ such that
\begin{equation}\label{eq:collar}
 u\in C^\infty\bigl(\overline{\B_1\setminus\B_{1-\delta}}\bigr),
 \qquad
 |u|>a\quad\text{there},
 \qquad
 \Delta u=0\quad\text{in }\B_1\setminus\overline{\B_{1-\delta}}.
\end{equation}
\end{proposition}

\begin{proof}
By \cref{lem:modulus-truncation}, $u$ takes values almost everywhere in the
bounded smooth closed annulus
\[
 K_{a,2}=\{y\in\R^n:a\le|y|\le2\}.
\]
Every competitor with values in $K_{a,2}$ is also admissible for the original
exterior-ball problem. Since $u$ itself takes values in $K_{a,2}$, it is a
Dirichlet-energy minimizer in this restricted class as well. Apply
Duzaar--Fuchs \cite[Theorem~II]{DuzaarFuchs1986} with
\[
 \Omega=\B_1,
 \qquad
 M=\{y\in\R^n:a<|y|<2\},
 \qquad
 v_0(x)=x\quad\text{on }\partial\B_1.
\]
Here $\Omega$ and $M$ have smooth boundary, $u$ minimizes among maps with
values in $\overline M=K_{a,2}$, and
$v_0\in C^\infty(\partial\B_1;M)$. The theorem therefore gives a representative
of $u$ that is continuous up to $\partial\B_1$ in a full boundary
neighborhood. Since $|u|=1>a$ pointwise on $\partial\B_1$, compactness of the
boundary and continuity provide $\delta>0$ such that
\[
 |u|>a
 \quad\text{in }\B_1\setminus\overline{\B_{1-\delta}}.
\]
Both positive and negative target variations are admissible in this collar,
and hence
\[
 \Delta u=0
 \quad\text{in }\B_1\setminus\overline{\B_{1-\delta}}.
\]
Classical boundary regularity for the harmonic Dirichlet problem with the
smooth boundary datum $x$ then gives the asserted $C^\infty$ regularity after
decreasing $\delta$ if necessary. This proves \eqref{eq:collar}.
\end{proof}

\begin{proposition}\label{prop:local-consequences}
Let $u$ be any local minimizer for the exterior-ball constraint. Then:
\begin{enumerate}[label=\textup{(\roman*)}]
\item $u$ satisfies \eqref{eq:EL}, \eqref{eq:stress}, and \eqref{eq:monotonicity};
\item $\Sing u$ is relatively closed and
$u\in C^{1,1}_{\mathrm{loc}}(\Omega\setminus\Sing u)$;
\item every $p\in\Sing u$ has a neighborhood on which $|u|=a$ almost everywhere;
\item the gap function $d_u=|u|-a$ has a continuous representative on the domain.
\end{enumerate}
\end{proposition}

\begin{proof}
Part (i) combines the Euler equation
\cite[equation~(1.2)]{FGKSsing}, the domain-stationarity computation
\eqref{eq:stress}, and the monotonicity formula
\cite[Lemma~2.1]{FGKSsing}. Part (ii) follows from the epsilon-regularity
theorem and the resulting singular-density lower bound
\cite[Theorem~2.4 and Corollary~2.5]{FGKSsing}, and part (iii) from the
coincidence-neighborhood consequence of \cite[Theorem~1.1]{FGKSsing}. For
(iv), let $p$ be arbitrary. If
$p\notin\Sing u$, then $u$ is continuous near $p$, so $d_u$ is continuous
there. If $p\in\Sing u$, the same coincidence-neighborhood result gives a ball
on which $|u|=a$ almost everywhere; define $d_u=0$ on that ball. On the
regular portion of the same ball, continuity of $u$ upgrades the
almost-everywhere equality to pointwise equality. These local representatives
agree on overlaps and provide a continuous representative of $d_u$.  In what
follows, the coincidence set is the closed set $\{d_u=0\}$ and the
noncontact set is the open set $\{d_u>0\}$.
\end{proof}

For a radial map $u(x)=w(r)e_r$, where $r=|x|$ and $e_r=x/r$, one has
\begin{equation}\label{eq:radial-gradient}
 |Du|^2=(w')^2+\frac{n-1}{r^2}w^2.
\end{equation}
Indeed, $D e_r=r^{-1}(I-e_r\otimes e_r)$, the radial and tangential pieces are orthogonal, and $|I-e_r\otimes e_r|^2=n-1$.

On a noncontact interval, the vector equation $\Delta(w e_r)=0$ becomes
\begin{equation}\label{eq:radial-ode}
 w''+\frac{n-1}{r}w'-\frac{n-1}{r^2}w=0.
\end{equation}
Its solutions are $c_1r+c_2r^{1-n}$.  On the contact interval the modulus is
constant, $w\equiv a$, so its radial derivative at the interface is zero.
The natural $C^1$ matching condition with the outer harmonic branch is
therefore $w'(R)=0$; equivalently, this condition prevents a jump of the
normal derivative, and hence a surface term in the distributional equation,
across $\partial\B_R$.  Imposing
\[
 w(R)=a,
 \qquad
 w'(R)=0,
 \qquad
 w(1)=1
\]
gives precisely \eqref{eq:aR} and \eqref{eq:radial-map}.

We next record the radial calculations that will be used in the parameter
argument.  For $r\ge R$,
\[
 w_R'(r)=\frac{n-1}{n-1+R^n}\bigl(1-R^nr^{-n}\bigr)\ge0.
\]
Thus $w_R(r)\ge w_R(R)=a$.  On $\B_R$, the map is
$a\frac{x}{|x|}$; since $n\ge3$, its energy density
$(n-1)a^2r^{-2}$ is integrable at the origin.  Hence
$u_{a,n}\in\cA_a$.

Set
\[
 M_n(a)=E(u_{a,n};\B_1),
 \qquad
 \beta_R=w_R'(1).
\]
Differentiating the outer expression for $w_R$ and using \eqref{eq:aR}
gives
\begin{equation}\label{eq:betaR}
 \beta_R
 =\frac{(n-1)(1-R^n)}{n-1+R^n}
 =1-aR^{n-1}.
\end{equation}
The same relation also yields
\[
 \beta_R+n-1=\frac{(n-1)a}{R}.
\]

The energy on the contact ball follows directly from
\eqref{eq:radial-gradient}:
\begin{equation}\label{eq:Crad}
 C_{\rm rad}:=E(u_{a,n};\B_R)
 =\frac{(n-1)\sigma_{n-1}}{n-2}a^2R^{n-2}.
\end{equation}
On the noncontact annulus, $u_{a,n}$ is harmonic.  Integration by parts,
together with $w_R'(R)=0$, $u_{a,n}=x$ on $\partial\B_1$, and
$\partial_\nu u_{a,n}=\beta_Rx$ there, gives
\begin{equation}\label{eq:Eouter}
 E(u_{a,n};\B_1\setminus\B_R)=\sigma_{n-1}\beta_R.
\end{equation}
Combining \eqref{eq:Crad}, \eqref{eq:Eouter}, \eqref{eq:aR}, and
\eqref{eq:betaR}, we obtain the energy relation
\begin{equation}\label{eq:radial-beta-energy}
 \beta_R^2=(n-1)-\frac{n-2}{\sigma_{n-1}}M_n(a).
\end{equation}

For the later parameter variation, differentiate along the parameter $R$.
Equation \eqref{eq:aR} gives
\[
 \frac{\dd a}{\dd R}
 =\frac{n(n-1)(1-R^n)}{(n-1+R^n)^2}>0.
\]
Differentiating $M_n=C_{\rm rad}+\sigma_{n-1}\beta_R$ and dividing by
$\dd a/\dd R$ yields
\begin{equation}\label{eq:Mprime}
 M_n'(a)=\frac{2C_{\rm rad}}a
 =\frac{2(n-1)\sigma_{n-1}}{n-2}aR^{n-2}.
\end{equation}
Finally, the two identities
$1-\beta_R=aR^{n-1}$ and
$\beta_R+n-1=(n-1)a/R$ give
\begin{equation}\label{eq:radial-Psi}
 \frac{(1-\beta_R)(\beta_R+n-1)^{n-1}}{a^n}
 =(n-1)^{n-1}.
\end{equation}

Define
\begin{equation}\label{eq:rho-kappa}
 \rho_n=\frac{n-1}{n},
 \qquad
 \theta_n=\frac{(n-1)\sigma_{n-1}}{n-2},
 \qquad
 \kappa_n=\frac{2\sigma_{n-1}}{n-2}\rho_n^{n-1}.
\end{equation}

We conclude the radial calculation by recording its small-$a$ expansion.
Rearranging \eqref{eq:aR},
\[
 R=\rho_na+\frac{a}{n}R^n.
\]
Since $R=O(a)$, it follows that
\[
 R(a)=\rho_na+O(a^{n+1}).
\]
By \eqref{eq:Crad}, \eqref{eq:Eouter}, and \eqref{eq:betaR},
\[
 M_n(a)-\sigma_{n-1}
 =\theta_na^2R^{n-2}-\sigma_{n-1}aR^{n-1}.
\]
Substituting the expansion of $R(a)$, the coefficient of $a^n$ is
\begin{align*}
 \theta_n\rho_n^{n-2}-\sigma_{n-1}\rho_n^{n-1}
 &=\sigma_{n-1}\rho_n^{n-2}
 \left(\frac{n-1}{n-2}-\rho_n\right)\\
 &=\frac{2\sigma_{n-1}}{n-2}\rho_n^{n-1}
 =\kappa_n.
\end{align*}
The error in $R$ is $O(a^{n+1})$, so its contribution to both terms is
$O(a^{2n})$. Therefore
\begin{equation}\label{eq:M-expansion}
 M_n(a)=\sigma_{n-1}+\kappa_na^n+O(a^{2n})
 \qquad(a\downarrow0).
\end{equation}

\section{The sharp infinite-cell problem}\label{sec:cell}

The proof of \cref{thm:main} begins in this section.  Since the radial map is
admissible, it already gives $m_n(a)\le M_n(a)$.  To obtain the reverse
inequality, we first determine the exact energetic cost of a small obstacle.
This is encoded by the infinite-cell problem below.  The main result of the
section is the exact cell constant in \cref{thm:sharp-cell}; it will determine
the endpoint asymptotic in \cref{sec:global-minimality}.

Let $D^{1,2}(\R^n;\R^n)$ denote the completion of
$C_c^\infty(\R^n;\R^n)$ in the norm $\|DZ\|_{L^2}$. Since $n\ge3$, the
Sobolev inequality identifies this completion with
\[
 \bigl\{Z\in L^{2^*}(\R^n;\R^n):DZ\in L^2(\R^n)\bigr\},
\]
where
\[
 2^*=\frac{2n}{n-2}.
\]
More precisely,
\begin{equation}\label{eq:sobolev-D12}
 \|Z\|_{L^{2^*}(\R^n)}\le C_n\|DZ\|_{L^2(\R^n)}.
\end{equation}
Define
\begin{align}
 c_{\infty,n}
 &=\inf\Bigl\{\cC_n(U):U=x+Z,\ Z\in D^{1,2}(\R^n;\R^n),\ |U|\ge1\text{ a.e.}\Bigr\},\label{eq:cinfty}\\
 \cC_n(U)&=\int_{\R^n}|D(U-x)|^2\dd x=\int_{\R^n}|DZ|^2\dd x.\label{eq:cell-energy}
\end{align}

\subsection{Cell minimizers and their behavior at infinity}

Set $\rho=\rho_n$ and
\[
 b_n=\frac{\rho^n}{n-1}.
\]
Define
\begin{equation}\label{eq:radial-cell}
 U_n^*(x)=
 \begin{cases}
 \displaystyle\frac{x}{|x|},&0<|x|\le\rho,\\[2mm]
 \displaystyle x+b_n\frac{x}{|x|^n},&|x|>\rho.
 \end{cases}
\end{equation}
The value at the origin may be chosen arbitrarily; it has no effect on the
Sobolev class or the energy.

A direct calculation shows that $U_n^*$ is admissible for
\eqref{eq:cinfty}, is $C^1$ across $\partial\B_\rho$, and satisfies
\begin{equation}\label{eq:cell-kappa}
 \cC_n(U_n^*)=\kappa_n.
\end{equation}
Indeed, for $r>\rho$, the modulus is
\[
 f(r)=r+b_nr^{1-n}.
\]
Since
\[
 f'(r)=1-(n-1)b_nr^{-n}=1-\rho^nr^{-n},
\]
$f$ has its minimum at $r=\rho$, and
\[
 f(\rho)=\rho+\frac{\rho}{n-1}=1.
\]
The inner modulus is also $1$, so $|U_n^*|\ge1$. At $r=\rho$, both the radial derivative and the tangential derivative match: the radial derivative is $0$, while the tangential factor is $1/\rho$. Thus $U_n^*$ is $C^1$ across the interface.

Put $Z_n^*=U_n^*-x$. Near the origin, $|DZ_n^*|\le C(1+r^{-1})$, which is square integrable because $n\ge3$; moreover $Z_n^*$ is bounded there. At infinity,
\[
 |DZ_n^*|\le Cr^{-n},
 \qquad
 |Z_n^*|=b_nr^{1-n}.
\]
Consequently $DZ_n^*\in L^2(\R^n)$ and $Z_n^*\in L^{2^*}(\R^n)$. To verify
the completion definition, choose radial cutoffs $\chi_R$ such that
$\chi_R=1$ on $\B_R$, $\chi_R=0$ outside $\B_{2R}$, and
$|D\chi_R|\le C/R$.  Then
\[
 D(\chi_RZ_n^*)-DZ_n^*
 =(\chi_R-1)DZ_n^*+Z_n^*\otimes D\chi_R.
\]
The $L^2$ norm of the first term tends to zero because
$DZ_n^*\in L^2$.  For the second term, H\"older's inequality and the
relation $1-2/2^*=2/n$ give
\[
 \int|Z_n^*|^2|D\chi_R|^2
 \le C\left(\int_{\{R<|x|<2R\}}|Z_n^*|^{2^*}\dd x\right)^{2/2^*}
 \longrightarrow0,
\]
where the factor $R^{-2}$ from $|D\chi_R|^2$ is cancelled by the
$R^2$ supplied by the measure of the annulus.  Thus
$\chi_RZ_n^*\to Z_n^*$ in the gradient norm.  Finally, each compactly
supported $W^{1,2}$ map $\chi_RZ_n^*$ can be mollified in $W^{1,2}$, proving
$Z_n^*\in D^{1,2}$ under the completion definition fixed above.

We compute the energy. In $\B_\rho$,
\[
 Z_n^*=e_r-x,
 \qquad
 DZ_n^*=\frac1r(I-e_r\otimes e_r)-I.
\]
Therefore
\begin{align*}
 \int_{\B_\rho}|DZ_n^*|^2
 &=\sigma_{n-1}\int_0^\rho
 \left(\frac{n-1}{r^2}-\frac{2(n-1)}r+n\right)r^{n-1}\dd r\\
 &=\sigma_{n-1}\left[
 \frac{n-1}{n-2}\rho^{n-2}-2\rho^{n-1}+\rho^n
 \right].
\end{align*}
For $r>\rho$,
\[
 DZ_n^*=b_nr^{-n}(I-n e_r\otimes e_r).
\]
The matrix has $n-1$ tangential eigenvalues $1$ and radial eigenvalue $1-n$, so
$|I-ne_r\otimes e_r|^2=n(n-1)$. Hence
\begin{align*}
 \int_{\R^n\setminus\B_\rho}|DZ_n^*|^2
 &=\sigma_{n-1}n(n-1)b_n^2\int_\rho^\infty r^{-n-1}\dd r\\
 &=\sigma_{n-1}(n-1)b_n^2\rho^{-n}
 =\frac{\sigma_{n-1}}{n-1}\rho^n.
\end{align*}
Adding the two pieces and using $n\rho=n-1$ gives
\begin{align*}
 \cC_n(U_n^*)
 &=\sigma_{n-1}\rho^{n-2}
 \left(\frac{n-1}{n-2}-\rho\right)\\
 &=\frac{2\sigma_{n-1}}{n-2}\rho^{n-1}
 =\kappa_n.
\end{align*}

\begin{proposition}\label{prop:cell-structure}
The infimum $c_{\infty,n}$ is attained. Every cell minimizer $U=x+Z$ has
the following properties.
\begin{enumerate}[label=\textup{(\roman*)}]
\item If $\Omega\Subset\R^n$ and $V-U\in W_0^{1,2}(\Omega;\R^n)$, with
$|V|\ge1$ almost everywhere, then
\[
 E(U;\Omega)\le E(V;\Omega).
\]
\item There exists $R_0>0$ such that $U$ is regular on
$\R^n\setminus\B_{R_0}$, the coincidence set is contained in $\B_{R_0}$,
and $Z$ is harmonic on $\R^n\setminus\overline{\B_{R_0}}$.
\item As $R\to\infty$,
\begin{equation}\label{eq:Z-decay}
 \sup_{|x|=R}|Z(x)|=o\bigl(R^{-(n-2)/2}\bigr).
\end{equation}
Consequently, for every sufficiently large $R$,
\begin{equation}\label{eq:outer-degree}
 \degmap\left(\left.\frac{U}{|U|}\right|_{\partial\B_R}\right)=1,
\end{equation}
and $U$ has at least one mapping singularity.
\end{enumerate}
\end{proposition}

\begin{proof}
Let $U_j=x+Z_j$ be a minimizing sequence. Then $(DZ_j)$ is bounded in $L^2$. By \eqref{eq:sobolev-D12}, $(Z_j)$ is bounded in $L^{2^*}$. Passing to a subsequence,
\[
 Z_j\weakto Z\quad\text{in }D^{1,2},
 \qquad
 Z_j\weakto Z\quad\text{in }L^{2^*}.
\]
For each fixed $R$, H\"older's inequality controls $\|Z_j\|_{L^2(\B_R)}$ by the $L^{2^*}$ norm, so $(Z_j)$ is bounded in $W^{1,2}(\B_R)$. A diagonal Rellich argument gives
\[
 Z_j\to Z\quad\text{strongly in }L^2_{\loc}(\R^n)
\]
and almost everywhere. Consequently,
\[
 |x+Z(x)|=\lim_{j\to\infty}|x+Z_j(x)|\ge1
\quad\text{for a.e. }x.
\]
Thus $U=x+Z$ is admissible. Weak lower semicontinuity yields
\[
 \cC_n(U)\le\liminf_{j\to\infty}\cC_n(U_j)=c_{\infty,n}.
\]

We first prove local minimality. Let $U=x+Z$ be a cell minimizer, and let
$V$ be as in part~(i).
Extend $W=V-U$ by zero outside $\Omega$. Then $Z+W\in D^{1,2}$, and the patched map is globally admissible. Since $W\in W_0^{1,2}(\Omega)$,
\[
 \int_\Omega I:DW\dd x=\int_\Omega\operatorname{div}W\dd x=0.
\]
Therefore
\begin{align*}
 \int_\Omega|D(V-x)|^2-|D(U-x)|^2
 &=\int_\Omega|DV|^2-|DU|^2-2\int_\Omega I:DW\\
 &=E(V;\Omega)-E(U;\Omega).
\end{align*}
Global cell minimality makes the left side nonnegative.

Monotonicity ensures the existence of the density
\[
 \Theta(U,x)=\lim_{r\downarrow0}\cE(U,x,r).
\]
Let $\eps_0>0$ denote the constant in the epsilon-regularity theorem
\cite[Theorem~2.4]{FGKSsing}.  In the present normalization, that theorem and
\cite[Corollary~2.5]{FGKSsing} give
\[
 \cE(U,x,2r)\le\eps_0^2
 \quad\Longrightarrow\quad U\in C^{1,1}(\B_r(x)),
 \qquad
 x\in\Sing U\quad\Longrightarrow\quad\Theta(U,x)\ge\eps_0^2.
\]

We next study the behavior of a cell minimizer at infinity. Fix a radius
$r_0>0$, to be chosen. By $DU=I+DZ$ and
$|A+B|^2\le2|A|^2+2|B|^2$,
\begin{align}
 \cE(U,x_0,2r_0)
 &\le 2(2r_0)^{2-n}\int_{\B_{2r_0}(x_0)}|DZ|^2\dd x
 +2n|\B_1|(2r_0)^2.\label{eq:far-energy}
\end{align}
Choose $r_0$ so small that the second term is at most $\eps_0^2/2$. Since $DZ\in L^2(\R^n)$,
\[
 \sup_{|x_0|\ge R}
 \int_{\B_{2r_0}(x_0)}|DZ|^2\dd x
 \le
 \int_{\{|x|\ge R-2r_0\}}|DZ|^2\dd x\longrightarrow0.
\]
Thus \eqref{eq:far-energy} is below $\eps_0^2$ for all sufficiently large
$|x_0|$. The estimate in \cite[Theorem~2.4]{FGKSsing} then gives, on the
smaller balls $\B_{r_0}(x_0)$, uniform bounds for $DU$ and $D^2U$ depending
only on $n$ and $r_0$, and not on $x_0$.  In particular, these estimates give
a uniform Lipschitz modulus there, and $U$ is regular outside a sufficiently
large ball.  This proves the first assertion in part~(ii).

Assume by contradiction that there are contact points $x_k$ with
$|x_k|\to\infty$. For all sufficiently large $k$, the map $U$ is regular
near $x_k$ by the preceding argument; for the continuous representative
fixed in \cref{prop:local-consequences}, the definition of the coincidence
set therefore gives $|U(x_k)|=1$. Taking a fixed
$0<\delta<r_0$ and using the preceding uniform gradient bound, we obtain a
constant $L<\infty$, independent of $k$, such that
\[
 |U(y)|\le1+L\delta
 \qquad\text{for }y\in\B_\delta(x_k).
\]
For sufficiently large $k$ and every such $y$,
\[
 |Z(y)|=|U(y)-y|
 \ge |y|-|U(y)|
 \ge |x_k|-\delta-(1+L\delta)
 \ge\frac{|x_k|}{2}.
\]
Hence
\[
 \int_{\B_\delta(x_k)}|Z|^{2^*}\dd x
 \ge c\delta^n|x_k|^{2^*}\longrightarrow\infty,
\]
contradicting $Z\in L^{2^*}(\R^n)$. Thus the contact set is bounded.

Outside a larger ball, $|U|>1$ and $U$ is regular. Two-sided variations are admissible there, so $\Delta U=0$. Since $\Delta x=0$, also $\Delta Z=0$.

It remains to establish the decay estimate and the degree assertion. By
part~(ii), each component of $Z$ is harmonic outside $\B_{R_0}$. Let $R$
be large and $x\in\partial\B_R$. The ball $\B_{R/4}(x)$ lies in the
harmonic region. The interior $L^p$ mean-value estimate, with $p=2^*$, gives
\begin{align*}
 |Z(x)|
 &\le C R^{-n/2^*}
 \|Z\|_{L^{2^*}(\B_{R/4}(x))}\\
 &\le C R^{-(n-2)/2}
 \|Z\|_{L^{2^*}(\{|y|>R/2\})}.
\end{align*}
The tail norm tends to zero, proving \eqref{eq:Z-decay}.

In particular, $|Z(x)|<|x|$ on $\partial\B_R$ for large $R$. The homotopy
\[
 H_t(x)=\frac{x+tZ(x)}{|x+tZ(x)|},\qquad 0\le t\le1,
\]
is well defined on $\partial\B_R$ and connects
$\frac{x}{|x|}$ to $\frac{U}{|U|}$. Thus the degree is one.

If $U$ had no mapping singularity, then regularity theory would make $U$
continuous on $\overline{\B_R}$. Because $|U|\ge1$, the normalized map
$\frac{U}{|U|}$ would continuously extend the boundary map to $\B_R$,
forcing its boundary degree to be zero. This contradicts
\eqref{eq:outer-degree}.
\end{proof}

\subsection{Singularities and the sharp density bound}\label{sec:singular}

Let $U$ be a cell minimizer.  The coincidence-neighborhood consequence of
\cite[Theorem~1.1]{FGKSsing}, after rescaling the obstacle radius to one,
states that every $p\in\Sing U$ has a ball on which $|U|=1$ almost everywhere.
Every sphere-valued competitor in that ball is also an admissible
exterior-ball competitor. Therefore $U$ minimizes the ordinary sphere-valued
Dirichlet energy there and is, in particular, a stable stationary harmonic
map into $\Sph^{n-1}$.

\begin{proposition}\label{prop:sharp-density}
For every $n\in\{3,4,5,6\}$ and every cell minimizer $U$, there exists
$p\in\Sing U$ such that
\begin{equation}\label{eq:sharp-density-point}
 \Theta(U,p)\ge\theta_n
 =\frac{(n-1)\sigma_{n-1}}{n-2}.
\end{equation}
\end{proposition}

\begin{proof}
We first treat the case $n=3$.

By \cref{prop:cell-structure}, $U$ has a singular point $p$. The
coincidence-neighborhood theorem makes $U$ an $\Sph^2$-valued energy
minimizer in some ball $\B_{r_*}(p)$.  For any sequence $r_j\downarrow0$,
the rescalings
\[
 U_j(x)=U(p+r_jx)
\]
are $\Sph^2$-valued energy minimizers on each fixed ball for all sufficiently
large $j$.  Monotonicity supplies uniform local gradient bounds, while
$|U_j|=1$ supplies the local $L^2$ bounds.  Hence the strong compactness
theorem \cite[Lemma~2.2]{FGKSsing}, followed by a diagonal argument, gives a
subsequence converging strongly in $W^{1,2}_{\mathrm{loc}}(\R^3)$ to a
tangent map $\Phi$.  For every fixed $s>0$, strong convergence and the
definition of the density give
\[
 s^{-1}\int_{\B_s}|D\Phi|^2\dd x
 =\lim_{j\to\infty}s^{-1}\int_{\B_s}|DU_j|^2\dd x
 =\lim_{j\to\infty}\cE(U,p,sr_j)
 =\Theta(U,p).
\]
Thus the normalized energy of $\Phi$ is constant in the radius.  The equality
case of \eqref{eq:monotonicity} gives $x\cdot D\Phi=0$ almost everywhere away
from the origin, so $\Phi$ is zero homogeneous. By the three-dimensional
partial-regularity and tangent-map classification results of
Brezis--Coron--Lieb
\cite[Section~VII, Theorems~7.3--7.4 and Corollary~7.12]{BCL1986}, $p$ is an
isolated singularity and every tangent map at $p$ has the form
\begin{equation}\label{eq:BCL-cell-tangent}
 \Phi(x)=Q\frac{x}{|x|},
 \qquad
 Q\in O(3)
\end{equation}
for $x\ne0$.  Strong convergence on $\B_1$ identifies the density with the
energy of this tangent map.  Hence
\begin{align*}
 \Theta(U,p)
 &=\int_{\B_1}|D\Phi|^2\dd x
 =\int_0^1\int_{\Sph^2}2\dd S\dd r
 =8\pi.
\end{align*}
Since $\theta_3=2\sigma_2=8\pi$, this proves \eqref{eq:sharp-density-point}
when $n=3$.

We now turn to the cases $n\in\{4,5,6\}$.  The argument has three steps: we
prove that the singular set is finite, select a singularity of nonzero local
degree, and estimate the energy of a tangent map at that singularity.

We first prove that the singular set is finite. For a stable stationary
harmonic map into $\Sph^k$,
Theorem~1 of Lin--Wang
\cite{LinWang2006} gives
\[
 \Hdim\Sing U\le n-\widehat d(k)-1.
\]
Here $k=n-1$ and
$\widehat d(3)=3$, $\widehat d(4)=4$, and $\widehat d(5)=5$; hence the
right-hand side is zero for $(n,k)=(4,3),(5,4),(6,5)$.  Thus
$\Hdim\Sing U\le0$ inside each coincidence neighborhood. This alone does not
imply local finiteness, so an additional blow-up argument is required.

Assume that distinct singular points $p_j$ converge to a point $p$. The regular set is open by epsilon regularity, so $p$ is singular. By the coincidence-neighborhood result cited above, there exists $r_*>0$ such that
\[
 U(\B_{r_*}(p))\subset\Sph^{n-1}
 \quad\text{a.e.}
\]
and $U$ is sphere-valued energy minimizing there. Put
\[
 r_j=|p_j-p|,
 \qquad
 U_j(x)=U(p+r_jx),
 \qquad
 q_j=\frac{p_j-p}{r_j}\in\Sph^{n-1}.
\]
Passing to a subsequence, $q_j\to q\in\Sph^{n-1}$.

For every fixed $R>0$ and all sufficiently large $j$, $U_j$ is a sphere-valued minimizer on $\B_R$. The monotonicity formula gives a uniform energy bound:
\begin{align*}
 \int_{\B_R}|DU_j|^2\dd x
 &=r_j^{2-n}\int_{\B_{Rr_j}(p)}|DU|^2\dd x\\
 &=R^{n-2}\cE(U,p,Rr_j)
 \le R^{n-2}\cE(U,p,r_*).
\end{align*}
By the strong compactness theorem \cite[Lemma~2.2]{FGKSsing}, after a diagonal subsequence,
\begin{equation}\label{eq:second-blowup-conv}
 U_j\to\Phi\qquad\text{strongly in }W^{1,2}_{\loc}(\R^n).
\end{equation}
The limit is a sphere-valued minimizing harmonic map on every bounded ball.

For every $s>0$, strong convergence and the definition of density give
\begin{align}
 s^{2-n}\int_{\B_s}|D\Phi|^2\dd x
 &=\lim_{j\to\infty}s^{2-n}\int_{\B_s}|DU_j|^2\dd x\notag\\
 &=\lim_{j\to\infty}\cE(U,p,sr_j)
 =\Theta(U,p).\label{eq:constant-density}
\end{align}
Thus the normalized energy of $\Phi$ is constant in $s$. Equality in the monotonicity formula implies
\[
 x\cdot D\Phi(x)=0\quad\text{for a.e. }x\ne0,
\]
so $\Phi$ is zero homogeneous.

Each $q_j$ is a singular point of $U_j$, so the singular-density lower bound gives
$\Theta(U_j,q_j)\ge\eps_0^2$.  The moving-center upper-semicontinuity
\cite[Corollary~2.3]{FGKSsing} asserts that
\begin{equation}\label{eq:usc}
 \Theta(\Phi,q)\ge\limsup_{j\to\infty}\Theta(U_j,q_j).
\end{equation}
Consequently,
\[
 \Theta(\Phi,q)\ge\limsup_{j\to\infty}\Theta(U_j,q_j)\ge\eps_0^2.
\]
If $q$ were regular, the $C^{1,1}$ regularity supplied by
\cite[Theorem~2.4]{FGKSsing} would give, for all sufficiently small $s$,
\[
 s^{2-n}\int_{\B_s(q)}|D\Phi|^2\dd x\le Cs^2\longrightarrow0,
\]
and hence $\Theta(\Phi,q)=0$, a contradiction. Thus $q\in\Sing\Phi$.

Zero homogeneity and scale invariance of regularity imply that every point $tq$, $t>0$, is singular. Hence $\Sing\Phi$ contains a ray and has Hausdorff dimension at least one. This contradicts the Lin--Wang estimate above, which gives $\Hdim\Sing\Phi\le0$ for the pairs
$(n,n-1)=(4,3),(5,4),(6,5)$.

Therefore singularities cannot accumulate.  The behavior at infinity in
\cref{prop:cell-structure} shows that the singular set is bounded.  It is
closed because the regular set is open.  An infinite closed bounded set in
$\R^n$ has an accumulation point, so the singular set is finite.

We next show that at least one singular point has nonzero local degree. Let
\[
 \Sing U=\{p_1,\dots,p_N\}.
\]
Choose pairwise disjoint balls $\overline{\B_{r_i}(p_i)}$ containing no other singularity. Define
\begin{equation}\label{eq:local-degree}
 d_i=\degmap\left(\left.\frac{U}{|U|}\right|_{\partial\B_{r_i}(p_i)}\right).
\end{equation}
The value is independent of the sufficiently small radius because the normalized map is continuous on the intervening punctured annulus.

For every sufficiently large $R$, set
\[
 \Omega_R=\B_R\setminus\bigcup_{i=1}^N\overline{\B_{r_i}(p_i)}.
\]
The map $v=\frac{U}{|U|}$ is continuous on $\overline{\Omega_R}$. Standard
additivity of degree, with the inner boundary spheres carrying the opposite
orientation, and \eqref{eq:outer-degree} give
\begin{equation}\label{eq:degree-sum}
 1=\degmap\left(\left.\frac{U}{|U|}\right|_{\partial\B_R}\right)
 =\sum_{i=1}^N d_i.
\end{equation}
Since the sum is one, some $d_i$ is nonzero.

Fix such a singularity $p$ and write its local degree as $d_p\ne0$.  The
cited coincidence result makes $U$ sphere valued near $p$;
monotonicity gives uniform local gradient bounds for the rescalings, while
sphere valuedness gives the corresponding local $L^2$ bounds.  The strong
compactness theorem \cite[Lemma~2.2]{FGKSsing} therefore supplies a sequence
$r_j\downarrow0$ such that
\[
 U_j(x)=U(p+r_jx)\to\Phi(x)
 \quad\text{strongly in }W^{1,2}_{\loc}(\R^n).
\]
We next show that
\begin{equation}\label{eq:tangent-form}
 \Phi(x)=\varphi\left(\frac{x}{|x|}\right),
 \qquad
 \varphi\in C^\infty(\Sph^{n-1};\Sph^{n-1}),
 \qquad
 \degmap\varphi=d_p\ne0.
\end{equation}
For every fixed ball, $U_j$ is sphere valued there for all sufficiently large
$j$. As in \eqref{eq:constant-density}, every tangent map is zero homogeneous.
Because $p$ is isolated, for all sufficiently small radii the local degree on
$\partial\B_r(p)$ equals $d_p$.

We first show that $\Phi$ has no singular point away from the origin. If $q\ne0$ were singular, zero homogeneity would make the entire ray $\{tq:t>0\}$ singular. This would contradict the inequality
$\Hdim\Sing\Phi\le0$ obtained above from Theorem~1 of Lin--Wang. Hence
$\Phi$ is smooth on every compact annulus, and zero homogeneity gives
\[
 \Phi(x)=\varphi\left(\frac{x}{|x|}\right)
\]
for a smooth map $\varphi:\Sph^{n-1}\to\Sph^{n-1}$. The sphere-valued harmonic map equation for $\Phi$ reduces to the harmonic map equation for $\varphi$ on $\Sph^{n-1}$.

It remains to justify preservation of degree. Strong $W^{1,2}$ convergence alone does not preserve degree on $\Sph^{n-1}$ when $n>3$. Let
\[
 A=\{x:1/2<|x|<2\}.
\]
Since $\Phi$ is smooth on a slightly larger annulus, for every
$x\in\overline A$ there is $s_x>0$ such that
$\B_{2s_x}(x)$ lies in that annulus and
\[
 (2s_x)^{2-n}\int_{\B_{2s_x}(x)}|D\Phi|^2\dd y<\frac{\eps_0^2}{4}.
\]
The balls $\B_{s_x}(x)$ cover $\overline A$; choose a finite subcover
$\{\B_{s_i}(x_i)\}_{i=1}^N$. Strong $W^{1,2}$ convergence transfers the
small-energy bound on every $\B_{2s_i}(x_i)$ to $U_j$ for all sufficiently
large $j$. On the concentric balls $\B_{s_i}(x_i)$,
\cite[Theorem~2.4]{FGKSsing} gives bounds for $DU_j$ and $D^2U_j$ that are
uniform in $j$. Since the cover is finite, the sequence is equicontinuous on
$\overline A$. After passing to a further subsequence, Arzel\`a--Ascoli and
the strong $L^2$ convergence identify the uniform limit as $\Phi$; hence
\begin{equation}\label{eq:uniform-annulus}
 U_j\to\Phi\qquad\text{uniformly on }\overline A.
\end{equation}
For large $j$, the maps $U_j|_{\Sph^{n-1}}$ and $\varphi$ are uniformly closer than $2$ in the Euclidean target sphere, so the shortest-geodesic interpolation gives a homotopy. Consequently,
\[
 \degmap\varphi
 =\degmap(U_j|_{\Sph^{n-1}})
 =\degmap(U|_{\partial\B_{r_j}(p)})
 =d_p\ne0.
\]

It remains to prove the required lower bound for the spherical energy. We
include the topological balancing argument rather than hiding it in the
citation to Ramanathan. More generally, if $m\ge1$ and
$\psi:\Sph^m\to\Sph^m$ is continuous with nonzero degree, then there is a
conformal automorphism $\gamma$ of $\Sph^m$ such that
\begin{equation}\label{eq:balance}
 \int_{\Sph^m}\psi\circ\gamma\dd S=0\in\R^{m+1}.
\end{equation}
For $\xi\in\B^{m+1}$, define the standard Mobius automorphism
\begin{equation}\label{eq:mobius}
 \Gamma_\xi(x)
 =\frac{(1-|\xi|^2)x+2(1+\xi\cdot x)\xi}
 {1+2\xi\cdot x+|\xi|^2},
 \qquad x\in\Sph^m.
\end{equation}
It is conformal, $\Gamma_0=\operatorname{id}$, and for each $p\in\Sph^m$,
\[
 \Gamma_{tp}(x)\longrightarrow p
 \quad\text{as }t\uparrow1
 \quad\text{for every }x\ne-p.
\]
Set
\[
 H(p,t)=\frac1{\sigma_m}\int_{\Sph^m}
 \psi(\Gamma_{tp}(x))\dd S_x,
 \qquad (p,t)\in\Sph^m\times[0,1).
\]
The map $H$ is continuous. Moreover,
\begin{equation}\label{eq:Hlimit}
 H(\cdot,t)\longrightarrow\psi
 \quad\text{uniformly on }\Sph^m\text{ as }t\uparrow1.
\end{equation}
To verify uniformity, fix $\eps>0$. Uniform continuity of $\psi$ provides
$\delta>0$ such that $|\psi(y)-\psi(p)|<\eps$ whenever
$d_{\Sph^m}(y,p)<\delta$. By rotational invariance, the measure of
\[
 \{x:d_{\Sph^m}(\Gamma_{tp}(x),p)\ge\delta\}
\]
is independent of $p$ and tends to zero as $t\uparrow1$. Splitting the integral over this exceptional set and its complement proves \eqref{eq:Hlimit}.

Suppose \eqref{eq:balance} fails for every conformal automorphism in this family. Then $H(p,t)\ne0$ for all $(p,t)$, and
\[
 \widetilde H(p,t)=\frac{H(p,t)}{|H(p,t)|}
\]
defines a continuous map into $\Sph^m$. At $t=0$, $H(p,0)$ is independent of
$p$, so $\widetilde H(\cdot,0)$ is constant and has degree zero. By
\eqref{eq:Hlimit}, $\widetilde H(\cdot,t)$ extends continuously to $t=1$ with
value $\psi$, so homotopy invariance gives $\degmap\psi=0$, a contradiction.
Hence $H(p,t)=0$ for some $(p,t)$, and $\gamma=\Gamma_{tp}$ satisfies
\eqref{eq:balance}.

Apply this balancing fact with $m=n-1$ and $\psi=\varphi$.  The map $\varphi$
is a smooth harmonic map from $\Sph^{n-1}$ to itself, and it is nonconstant
because its degree is nonzero.  Since $n-1\ge3$, Ramanathan's Main
Theorem~(a) \cite[p.~784]{Ramanathan1986} applies: it states that the energy
of such a harmonic map is no smaller than the energy of any of its
precompositions by an orientation-preserving conformal automorphism.  The
map $\gamma=\Gamma_{tp}$ found above lies in this conformal group because
$t\mapsto\Gamma_{tp}$ is a path from the identity. After removing
Ramanathan's factor $1/2$ from the energy normalization, the theorem gives
\[
 \int_{\Sph^{n-1}}|D\varphi|^2\dd S
 \ge
 \int_{\Sph^{n-1}}|D(\varphi\circ\gamma)|^2\dd S.
\]
The first positive eigenvalue of $-\Delta_{\Sph^{n-1}}$ is $n-1$.
The Poincare inequality applied componentwise, together with
\eqref{eq:balance}, therefore yields
\begin{equation}\label{eq:sphere-lower}
 \int_{\Sph^{n-1}}|D\varphi|^2\dd S
 \ge(n-1)\int_{\Sph^{n-1}}|\varphi\circ\gamma|^2\dd S
 =(n-1)\sigma_{n-1}.
\end{equation}

It remains to convert this sphere energy into the density at $p$. Strong
convergence and zero homogeneity give
\[
 \Theta(U,p)=\int_{\B_1}|D\Phi|^2\dd x.
\]
Using \eqref{eq:tangent-form} and polar coordinates, we obtain
\begin{align*}
 \Theta(U,p)
 &=\int_0^1r^{n-3}\dd r
 \int_{\Sph^{n-1}}|D_{\Sph^{n-1}}\varphi|^2\dd S\\
 &=\frac1{n-2}\int_{\Sph^{n-1}}|D\varphi|^2\dd S.
\end{align*}
Thus \eqref{eq:sphere-lower} gives
\begin{equation}\label{eq:charged-density}
 \Theta(U,p)\ge\theta_n
 =\frac{(n-1)\sigma_{n-1}}{n-2}.
\end{equation}
This is \eqref{eq:sharp-density-point} for $n\in\{4,5,6\}$ and completes
the proof.
\end{proof}

\subsection{Translated calibration and the sharp cell constant}\label{sec:calibration}

For $p\in\R^n$, define
\[
 h_p(x)=\frac{x-p}{|x-p|},
 \qquad
 H_p(x)=p+U_n^*(x-p).
\]
The field $H_p$ need not satisfy $|H_p|\ge1$ when $p\ne0$. It is used only as an algebraic reference field; admissibility is never asserted or needed.

\begin{theorem}\label{thm:sharp-cell}
For $n\in\{3,4,5,6\}$,
\begin{equation}\label{eq:cinfty-kappa}
 c_{\infty,n}=\kappa_n.
\end{equation}
\end{theorem}

\begin{proof}
We first establish the translated calibration identity. Let $U=x+Z$ be
admissible in \eqref{eq:cinfty}, and fix $p\in\R^n$. We claim that
\begin{equation}\label{eq:translated-identity}
 \cC_n(U)-\kappa_n
 =\int_{\B_{\rho_n}(p)}
 \bigl(|DU|^2-|Dh_p|^2\bigr)\dd x
 +\int_{\R^n\setminus\B_{\rho_n}(p)}
 |D(U-H_p)|^2\dd x.
\end{equation}
Write $\rho=\rho_n$ and
\[
 Z_p^*=H_p-x,
 \qquad
 W=Z-Z_p^*.
\]
Then $W\in D^{1,2}$. Define
\begin{equation}\label{eq:Ap}
 A_p=DZ_p^*-\1_{\B_\rho(p)}Dh_p.
\end{equation}
Let $y=x-p$, $r=|y|$, and $\nu=y/r$. From the definition of $U_n^*$,
\begin{equation}\label{eq:Ap-pieces}
 A_p=
 \begin{cases}
 -I,&r<\rho,\\[1mm]
 b_nr^{-n}(I-n\nu\otimes\nu),&r>\rho.
 \end{cases}
\end{equation}
We verify the divergence claim row by row.  For $r>0$, write
$B_{ij}=r^{-n}(\delta_{ij}-n\nu_i\nu_j)$.  Since
\[
 \partial_j(r^{-n}\delta_{ij})=-nr^{-n-1}\nu_i,
 \qquad
 \partial_j(r^{-n}\nu_i\nu_j)=-r^{-n-1}\nu_i,
\]
where $\nu_j\partial_j\nu_i=0$ and
$\operatorname{div}\nu=(n-1)/r$, we obtain
\[
 \partial_jB_{ij}=0.
\]
Thus each row of $A_p$ is divergence free in both regions in
\eqref{eq:Ap-pieces}. Across the interface, the two normal traces are
\[
 A_p^-\nu=-\nu,
 \qquad
 A_p^+\nu=-(n-1)b_n\rho^{-n}\nu=-\nu,
\]
because $b_n=\rho^n/(n-1)$. Thus the jump of the normal trace is zero, so
there is no interface measure in $\operatorname{div}A_p$. There is no point
mass at $p$, because $A_p=-I$ in a full neighborhood of $p$. Also
$A_p\in L^2(\R^n)$: it is bounded in the inner ball and decays like $r^{-n}$
outside. Hence
\begin{equation}\label{eq:A-divfree}
 \operatorname{div}A_p=0\quad\text{in }\cD'(\R^n),
 \qquad
 \int_{\R^n}A_p:DW\dd x=0.
\end{equation}
The second identity is first true for compactly supported smooth $W$ and then
follows for every $W\in D^{1,2}$ by approximation in the gradient norm and
$A_p\in L^2$.

Since $\cC_n(H_p)=\cC_n(U_n^*)=\kappa_n$ by translation invariance,
\begin{align*}
 \cC_n(U)-\kappa_n
 &=\int_{\R^n}|D(Z_p^*+W)|^2-|DZ_p^*|^2\dd x\\
 &=\int_{\R^n}|DW|^2\dd x
   +2\int_{\R^n}DZ_p^*:DW\dd x.
\end{align*}
By \eqref{eq:Ap}--\eqref{eq:A-divfree},
\[
 \int_{\R^n}DZ_p^*:DW
 =\int_{\B_\rho(p)}Dh_p:DW.
\]
Inside $\B_\rho(p)$, $DH_p=Dh_p$, and hence
\[
 DW=D(U-H_p)=DU-Dh_p.
\]
Therefore
\begin{align*}
 |DW|^2+2Dh_p:DW
 &=|DU-Dh_p|^2+2Dh_p:(DU-Dh_p)\\
 &=|DU|^2-|Dh_p|^2.
\end{align*}
Outside the ball, the term is simply $|D(U-H_p)|^2$. This proves
\eqref{eq:translated-identity} and, in particular, shows explicitly why no
boundary term at infinity is present.

We now apply this identity to a cell minimizer. Let $U$ be such a minimizer,
whose existence is given by
\cref{prop:cell-structure}. By \cref{prop:sharp-density}, choose a singularity
$p$ satisfying \eqref{eq:sharp-density-point}. Monotonicity centered at $p$
gives
\begin{align*}
 \int_{\B_{\rho_n}(p)}|DU|^2\dd x
 &\ge \Theta(U,p)\rho_n^{n-2}\\
 &\ge \theta_n\rho_n^{n-2}
 =\int_{\B_{\rho_n}(p)}|Dh_p|^2\dd x,
\end{align*}
where the second inequality is \eqref{eq:sharp-density-point}, and the final equality follows from $|Dh_p|^2=(n-1)|x-p|^{-2}$. The translated identity \eqref{eq:translated-identity} now yields
\[
 \cC_n(U)\ge\kappa_n.
\]
On the other hand, the radial cell $U_n^*$ is admissible and has energy
$\kappa_n$ by \eqref{eq:cell-kappa}. Hence equality holds.
\end{proof}

\section{From the endpoint asymptotic to global minimality}\label{sec:global-minimality}

The sharp cell constant now determines the behavior as $a\downarrow0$.  We
then propagate this endpoint information to every $a\in(0,1)$ by a parameter
monotonicity argument that applies to arbitrary minimizers.

\subsection{The small-obstacle asymptotic}\label{sec:finite-cell}

For $L\ge1$, define
\begin{equation}\label{eq:cL}
 c_{L,n}=\inf\left\{
 \int_{\B_L}|DZ|^2\dd x:
 Z\in W_0^{1,2}(\B_L;\R^n),\ |x+Z|\ge1\text{ a.e.}
 \right\}.
\end{equation}

\begin{proposition}\label{prop:endpoint}
For $n\in\{3,4,5,6\}$,
\begin{equation}\label{eq:endpoint}
 m_n(a)=\sigma_{n-1}+\kappa_na^n+o(a^n)
 \qquad(a\downarrow0).
\end{equation}
More precisely, if $L=a^{-1}$, then
\begin{equation}\label{eq:exact-scaling}
 m_n(a)-\sigma_{n-1}=a^n c_{L,n}.
\end{equation}
\end{proposition}

\begin{proof}
Given $u\in\cA_a$, define on $\B_L$
\[
 U(y)=a^{-1}u(ay),
 \qquad
 Z(y)=U(y)-y.
\]
Then $|U|\ge1$ and $U(y)=y$ on $\partial\B_L$, so $Z\in W_0^{1,2}(\B_L)$. Conversely, this construction is reversible. Since $D_yU(y)=D_xu(ay)$,
\begin{align*}
 E(u;\B_1)
 &=a^n\int_{\B_L}|DU|^2\dd y\\
 &=a^n\int_{\B_L}|I+DZ|^2\dd y\\
 &=a^n\left(n|\B_L|+2\int_{\B_L}\operatorname{div} Z\dd y+\int_{\B_L}|DZ|^2\dd y\right).
\end{align*}
The divergence integral vanishes because $Z$ has zero trace. Also
$a^nn|\B_L|=n|\B_1|=\sigma_{n-1}$. Taking infima proves \eqref{eq:exact-scaling}.

It remains to determine the limit of $c_{L,n}$.  Let $Z$ be admissible for
$c_{L,n}$.
Extend $Z$ by zero outside $\B_L$. The
extension belongs to $D^{1,2}$. On $\R^n\setminus\B_L$, the corresponding
map is $U(x)=x$, and because $L\ge1$, $|U(x)|=|x|\ge1$. Thus the extension is
admissible for the infinite-cell problem, and
\[
 c_{L,n}\ge c_{\infty,n}=\kappa_n.
\]

For the reverse bound, put $a=L^{-1}$ and apply the preceding scaling to the
radial map $u_{a,n}$.  More explicitly, define
\[
 U_L(y)=L u_{L^{-1},n}(y/L),
 \qquad
 Z_L(y)=U_L(y)-y
 \quad\text{in }\B_L.
\]
Because $u_{L^{-1},n}=x$ on $\partial\B_1$, the map $Z_L$ has zero trace on
$\partial\B_L$; moreover $|U_L|\ge1$.  Thus $Z_L$ is admissible for
$c_{L,n}$, and the same change of variables used in
\eqref{eq:exact-scaling} gives
\[
 c_{L,n}\le L^n\bigl(M_n(L^{-1})-\sigma_{n-1}\bigr).
\]
By \eqref{eq:M-expansion}, the right side tends to $\kappa_n$. Hence
$c_{L,n}\to\kappa_n$ as $L\to\infty$. By \eqref{eq:exact-scaling}, this is
exactly \eqref{eq:endpoint}.
\end{proof}

\subsection{Parameter monotonicity and global energy equality}
\label{sec:global-monotonicity}

Fix $a\in(0,1)$ and let $u$ be any minimizer at parameter $a$. Set
\begin{equation}\label{eq:ECP}
 E=\int_{\B_1}|Du|^2\dd x=m_n(a),
 \qquad
 C=\int_{\{|u|=a\}}|Du|^2\dd x,
 \qquad
 P=\int_{\partial\B_1}|\partial_\nu u|^2\dd S.
\end{equation}
The boundary integral is classical because of \cref{prop:boundary-collar}.

For $a,b\in(0,1)$ define
\begin{align}
 q_{a,b}(r)&=r+\frac{b-a}{1-a}(1-r),\label{eq:q-ab}\\
 T_{a,b}(y)&=q_{a,b}(|y|)\frac{y}{|y|}.
\end{align}
Since
\begin{equation}\label{eq:q-obstacle}
 q_{a,b}(r)-b=(r-a)\frac{1-b}{1-a}\ge0
 \qquad(r\ge a),
\end{equation}
$T_{a,b}$ maps the $a$-admissible target into the $b$-admissible target and fixes $\Sph^{n-1}$.
On every set $\{|y|\ge a_0\}$ with $a_0>0$, this radial formula is
Lipschitz.  Whenever the Sobolev chain rule is used below, we first extend
its restriction to the relevant closed annulus to a Lipschitz map on
$\R^n$.

\begin{proposition}\label{prop:variational-identities}
The function $m_n$ is locally Lipschitz on $(0,1)$.  With $E$, $C$, and $P$
defined by \eqref{eq:ECP}, every minimizer satisfies the following identities.
At every point $a$ where $m_n$ is differentiable,
\begin{equation}\label{eq:mprime-C}
 m_n'(a)=\frac{2C}{a}.
\end{equation}
Moreover,
\begin{equation}\label{eq:Pohozaev}
 P=(n-1)\sigma_{n-1}-(n-2)E,
\end{equation}
and
\begin{equation}\label{eq:testing-identity}
 E-C=\int_{\partial\B_1}x\cdot\partial_\nu u\dd S.
\end{equation}
Consequently,
\begin{equation}\label{eq:C-lower}
 C\ge E-
 \sqrt{\sigma_{n-1}\bigl((n-1)\sigma_{n-1}-(n-2)E\bigr)}.
\end{equation}
\end{proposition}

\begin{proof}
Fix a compact interval $I\Subset(0,1)$. For a radial map
$T(y)=q(|y|)\frac{y}{|y|}$, the radial eigenvalue of $DT$ is $q'(|y|)$ and
each tangential eigenvalue is $\frac{q(|y|)}{|y|}$. From \eqref{eq:q-ab},
uniformly for $a,b\in I$ and $|y|\ge a$,
\[
 \|DT_{a,b}(y)-I\|\le C_I|b-a|.
\]
Applying $T_{a,b}$ to an $a$-minimizer gives
\[
 m_n(b)\le (1+C_I|b-a|)^2m_n(a).
\]
Applying $T_{b,a}$ to a $b$-minimizer gives the reverse estimate. Since $m_n$ is locally bounded above by the radial energy, these inequalities imply local Lipschitz continuity.

Now fix an $a$-minimizer $u$ and set
\begin{equation}\label{eq:Fa}
 F_a(y)=\frac{1-|y|}{1-a}\frac{y}{|y|}.
\end{equation}
On the range $a\le|u|\le1$, the function $F_a$ is Lipschitz; one may extend it Lipschitzly to all of $\R^n$. Moreover, $F_a\circ u\in W_0^{1,2}(\B_1;\R^n)$ because $F_a=0$ on $\Sph^{n-1}$. Put
\[
 A=\int_{\B_1}Du:D(F_a\circ u)\dd x,
 \qquad
 B=\int_{\B_1}|D(F_a\circ u)|^2\dd x.
\]
For $t>0$, $T_{a,a+t}\circ u$ is admissible at $a+t$, so
\[
 \frac{m_n(a+t)-m_n(a)}t
 \le 2A+tB.
\]
For $t<0$, the same comparison still gives
$m_n(a+t)-m_n(a)\le2tA+t^2B$, but division by the negative number $t$ reverses the inequality:
\[
 \frac{m_n(a+t)-m_n(a)}t
 \ge2A+tB.
\]
If $m_n$ is differentiable at $a$, the two one-sided limits yield
\begin{equation}\label{eq:derivative-A}
 m_n'(a)=2\int_{\B_1}Du:D(F_a\circ u)\dd x.
\end{equation}

Use $F_a\circ u$ as a test field in the weak Euler equation \eqref{eq:EL}. This extension of the distributional identity is legitimate: $F_a\circ u\in W_0^{1,2}\cap L^\infty$, and it admits compactly supported smooth approximations converging in $W^{1,2}$ and almost everywhere with a common $L^\infty$ bound. The left side passes by $L^2$ convergence, while the right side passes by dominated convergence because $|Du|^2\1_{\{|u|=a\}}\in L^1$. On the contact set, $|u|=a$ and
\[
 F_a(u)=\frac{u}{a}.
\]
Therefore
\begin{align*}
 \int_{\B_1}Du:D(F_a\circ u)\dd x
 &=\int_{\{|u|=a\}}\frac{|Du|^2}{a^2}
 u\cdot\frac{u}{a}\dd x\\
 &=\frac1a\int_{\{|u|=a\}}|Du|^2\dd x
 =\frac Ca.
\end{align*}
Substitute this into \eqref{eq:derivative-A}.

We next derive the Pohozaev identity from the stationarity formula. Testing
\eqref{eq:stress} with radial cutoffs gives, for almost every
$0<r<1$,
\begin{equation}\label{eq:pohozaev-r}
 (n-2)\int_{\B_r}|Du|^2\dd x
 =r\int_{\partial\B_r}\left(|Du|^2-2|\partial_ru|^2\right)\dd S.
\end{equation}
For completeness, take $X(x)=\eta(|x|)x$ in \eqref{eq:stress}. Then
\[
 \partial_iX_j=\eta\delta_{ij}+\eta'\frac{x_ix_j}{|x|},
\]
so stationarity becomes
\[
 \int_{\B_1}(n-2)\eta|Du|^2\dd x
 +\int_{\B_1}|x|\eta'\bigl(|Du|^2-2|\partial_ru|^2\bigr)\dd x=0.
\]
Here is the limiting argument. Fix $r\in(0,1)$ and
$0<\varepsilon<1-r$, and define
\[
 \eta_{r,\varepsilon}(s)=
 \begin{cases}
 1,&0\le s\le r,\\[1mm]
 \displaystyle\frac{r+\varepsilon-s}{\varepsilon},
    &r<s<r+\varepsilon,\\[2mm]
 0,&r+\varepsilon\le s\le1.
 \end{cases}
\]
Smooth radial approximations with uniformly bounded derivatives converge to
$\eta_{r,\varepsilon}$ pointwise with almost-everywhere convergence of their
derivatives. Since $|Du|^2\in L^1$, dominated convergence therefore permits
$\eta_{r,\varepsilon}$ itself in the stationarity identity. If
\[
 F(s)=\int_{\partial\B_s}
 \bigl(|Du|^2-2|\partial_ru|^2\bigr)\dd S,
\]
the coarea formula gives the complete identity
\begin{align*}
 0={}&(n-2)\int_{\B_r}|Du|^2\dd x\\
 &+(n-2)\int_{\B_{r+\varepsilon}\setminus\B_r}
 \frac{r+\varepsilon-|x|}{\varepsilon}|Du|^2\dd x
 -\frac1\varepsilon\int_r^{r+\varepsilon}sF(s)\dd s.
\end{align*}
The second term tends to zero by absolute continuity of the integral. At
every Lebesgue point of the locally integrable function $s\mapsto sF(s)$,
\[
 \frac1\varepsilon\int_r^{r+\varepsilon}sF(s)\dd s
 \longrightarrow rF(r).
\]
Letting $\varepsilon\downarrow0$ therefore gives
\eqref{eq:pohozaev-r}. This also explains why the identity is initially
asserted only for almost every $r$.

Choose a sequence of such Lebesgue points $r_j\uparrow1$.  By
\cref{prop:boundary-collar}, $u$ is smooth in a full collar of the boundary,
so $F(r_j)\to F(1)$, while absolute continuity of the integral gives
$\int_{\B_{r_j}}|Du|^2\dd x\to E$.  On $\partial\B_1$, the trace is
$u(x)=x$. Hence the tangential derivative is the identity on the tangent
space and
\[
 |D_\tau u|^2=n-1,
 \qquad
 |Du|^2=n-1+|\partial_\nu u|^2.
\]
Letting $j\to\infty$ in \eqref{eq:pohozaev-r} gives
\[
 (n-2)E
 =\int_{\partial\B_1}\bigl(n-1-|\partial_\nu u|^2\bigr)\dd S
 =(n-1)\sigma_{n-1}-P.
\]
No free-boundary flux has been omitted: the derivation starts from the distributional stationarity identity, which already contains the entire constrained variational structure.

It remains to establish the testing identity and the resulting boundary
estimate. For small $\varepsilon>0$, let $\eta_\varepsilon$ be the radial
Lipschitz function that is $1$ on $[0,1-\varepsilon]$ and equals
$(1-r)/\varepsilon$ on $[1-\varepsilon,1]$.  It belongs to
$W_0^{1,\infty}(\B_1)$, and its transition region is contained in the smooth
noncontact collar.  Test \eqref{eq:EL} with $\eta_\varepsilon u$. Since
$|u|\le1$, the test lies in $W_0^{1,2}\cap L^\infty$; the
bounded-approximation argument used above for \eqref{eq:mprime-C} extends
\eqref{eq:EL} to it. In weak form,
\begin{align*}
 \int_{\B_1}\eta_\varepsilon|Du|^2\dd x
 +\int_{\B_1}u\cdot Du\,\nabla\eta_\varepsilon\dd x
 &=\int_{\{|u|=a\}}\eta_\varepsilon|Du|^2\dd x.
\end{align*}
Thus
\begin{equation}\label{eq:test-cutoff}
 \int_{\B_1}\eta_\varepsilon
 \bigl(|Du|^2-\1_{\{|u|=a\}}|Du|^2\bigr)\dd x
 =-\int_{\B_1}u\cdot Du\,\nabla\eta_\varepsilon\dd x.
\end{equation}
The left side tends to $E-C$ by dominated convergence.  To identify the
right side, set
\[
 G(r)=\int_{\partial\B_r}u\cdot\partial_ru\dd S.
\]
The smoothness in the collar makes $G$ continuous near $r=1$.  The coarea
formula gives
\[
 -\int_{\B_1}u\cdot Du\,\nabla\eta_\varepsilon\dd x
 =\frac1\varepsilon\int_{1-\varepsilon}^1G(r)\dd r,
\]
Consequently the right side of \eqref{eq:test-cutoff} tends to $G(1)$, namely
\[
 \int_{\partial\B_1}u\cdot\partial_\nu u\dd S
 =\int_{\partial\B_1}x\cdot\partial_\nu u\dd S.
\]
This proves \eqref{eq:testing-identity}.

Cauchy--Schwarz on the boundary gives
\[
 E-C
 \le\left(\int_{\partial\B_1}|x|^2\dd S\right)^{1/2}
     \left(\int_{\partial\B_1}|\partial_\nu u|^2\dd S\right)^{1/2}
 =\sqrt{\sigma_{n-1}P}.
\]
Use \eqref{eq:Pohozaev} to substitute for $P$ and rearrange.
\end{proof}

We now use these identities to propagate the endpoint asymptotic through the
whole interval $a\in(0,1)$. For every admissible map,
$u-x\in W_0^{1,2}(\B_1)$, and hence
\begin{align}
 \int_{\B_1}|Du|^2\dd x
 &=\int_{\B_1}|I+D(u-x)|^2\dd x\notag\\
 &=\sigma_{n-1}+\int_{\B_1}|D(u-x)|^2\dd x,
 \label{eq:energy-splitting}
\end{align}
because the cross term is $2\int\operatorname{div}(u-x)=0$. Equality with $\sigma_{n-1}$ would force $u=x$ almost everywhere, which violates $|u|\ge a$ on $\B_a$. Therefore
\begin{equation}\label{eq:m-greater-sigma}
 m_n(a)>\sigma_{n-1}
 \qquad(a>0).
\end{equation}

Define the positive square root
\begin{equation}\label{eq:beta-def}
 \beta(a)^2=(n-1)-\frac{n-2}{\sigma_{n-1}}m_n(a)
\end{equation}
and
\begin{equation}\label{eq:Psi-def}
 \Psi_n(a)=\frac{(1-\beta(a))(\beta(a)+n-1)^{n-1}}{a^n}.
\end{equation}
By \eqref{eq:Pohozaev}, for any minimizer
\begin{equation}\label{eq:beta-P}
 \beta(a)^2=\frac{P}{\sigma_{n-1}}.
\end{equation}
Since $m_n(a)\le M_n(a)$,
\begin{equation}\label{eq:beta-bounds}
 0<\beta_R(a)\le\beta(a)<1.
\end{equation}
Indeed, the first two inequalities follow from comparison with the radial branch, and the last follows from \eqref{eq:m-greater-sigma}.

\begin{theorem}\label{thm:propagation}
For every $n\in\{3,4,5,6\}$ and every $a\in(0,1)$,
\begin{equation}\label{eq:m-equals-M}
 m_n(a)=M_n(a).
\end{equation}
\end{theorem}

\begin{proof}
We first prove that $\Psi_n$ is locally absolutely continuous and
nondecreasing on $(0,1)$ and that
\begin{equation}\label{eq:Psi-upper}
 \Psi_n(a)\le(n-1)^{n-1}
 \qquad(0<a<1).
\end{equation}
Fix $I\Subset(0,1)$.  By \eqref{eq:beta-bounds} and the continuity of the
radial branch, $\beta\ge\beta_R\ge c_I>0$ on $I$.  Since $m_n$ is Lipschitz
there, the square-root relation \eqref{eq:beta-def} then implies that $\beta$
is absolutely continuous on $I$.  Moreover, $\beta<1$ and continuity give
$1-\beta\ge d_I>0$ on $I$.  All the factors in \eqref{eq:Psi-def}, including
their logarithms, are therefore Lipschitz on the relevant ranges.  Hence
$\Psi_n$ and $\log\Psi_n$ are absolutely continuous on $I$.

At almost every $a\in I$, the function $m_n$ is differentiable.  Choose any
minimizer at that parameter.  Differentiating \eqref{eq:beta-def} and using
\eqref{eq:mprime-C} gives
\begin{equation}\label{eq:beta-derivative}
 -\beta\beta'
 =\frac{n-2}{\sigma_{n-1}a}C.
\end{equation}
The boundary lower bound \eqref{eq:C-lower}, together with $P=\sigma_{n-1}\beta^2$, gives
\begin{align*}
 C
 &\ge E-\sigma_{n-1}\beta\\
 &=\frac{\sigma_{n-1}}{n-2}
 \bigl(n-1-\beta^2-(n-2)\beta\bigr)\\
 &=\frac{\sigma_{n-1}}{n-2}(1-\beta)(\beta+n-1).
\end{align*}
Substitution in \eqref{eq:beta-derivative} yields, at almost every such $a$,
\begin{equation}\label{eq:beta-diff-ineq}
 -\beta\beta'
 \ge\frac{(1-\beta)(\beta+n-1)}a.
\end{equation}
Since $0<\beta<1$, all denominators below are positive.  Taking the logarithm
of \eqref{eq:Psi-def} and differentiating term by term gives
\begin{align}
 \frac{\dd}{\dd a}\log\Psi_n
 &=-\frac{\beta'}{1-\beta}
 +(n-1)\frac{\beta'}{\beta+n-1}-\frac na\notag\\
 &=-\frac{n\beta\beta'}{(1-\beta)(\beta+n-1)}-\frac na
 \ge0\label{eq:Psi-derivative}
\end{align}
for almost every $a$, where the last inequality is exactly
\eqref{eq:beta-diff-ineq}.  Because $I$ was arbitrary and $\log\Psi_n$ is
absolutely continuous on $I$, the almost-everywhere differential inequality
implies that $\log\Psi_n$, and hence $\Psi_n$, is nondecreasing on $(0,1)$.

For the upper bound, set
\[
 f(\beta)=(1-\beta)(\beta+n-1)^{n-1}.
\]
On $(0,1)$,
\[
 f'(\beta)=-n\beta(\beta+n-1)^{n-2}<0.
\]
By \eqref{eq:beta-bounds}, $\beta\ge\beta_R$, and therefore
\[
 \Psi_n(a)=\frac{f(\beta(a))}{a^n}
 \le\frac{f(\beta_R(a))}{a^n}
 =(n-1)^{n-1}
\]
by \eqref{eq:radial-Psi}.

We now combine this monotonicity with the endpoint asymptotic.
The endpoint asymptotic \eqref{eq:endpoint} gives
\[
 \beta(a)^2
 =1-\frac{n-2}{\sigma_{n-1}}\kappa_na^n+o(a^n),
\]
so
\begin{equation}\label{eq:one-minus-beta}
 \frac{1-\beta(a)}{a^n}
 \longrightarrow\frac{(n-2)\kappa_n}{2\sigma_{n-1}}.
\end{equation}
Since $\beta(a)+n-1\to n$,
\begin{align*}
 \lim_{a\downarrow0}\Psi_n(a)
 &=n^{n-1}\frac{(n-2)\kappa_n}{2\sigma_{n-1}}\\
 &=n^{n-1}\rho_n^{n-1}
 =(n-1)^{n-1}.
\end{align*}
Since $\Psi_n$ is nondecreasing, for every $a>0$ it is at least its right-hand
limit at the left endpoint $0$. The opposite inequality is
\eqref{eq:Psi-upper}. Therefore
\[
 \Psi_n(a)=(n-1)^{n-1}
 \qquad\forall a\in(0,1).
\]
By the definition of $\Psi_n$, the equality just obtained and
\eqref{eq:radial-Psi} give, more explicitly,
\[
 \frac{f(\beta(a))}{a^n}
 =(n-1)^{n-1}
 =\frac{f(\beta_R(a))}{a^n}.
\]
Thus $f(\beta(a))=f(\beta_R(a))$.  Since $f$ is strictly decreasing on
$(0,1)$, it follows that $\beta(a)=\beta_R(a)$. The definition
\eqref{eq:beta-def} then gives $m_n(a)=M_n(a)$.
\end{proof}

\section{Rigidity and uniqueness}\label{sec:uniqueness}

By \cref{thm:propagation}, the radial map and every minimizer have the same
energy.  It remains to determine the equality cases.  We first recover the
normal derivative on the outer boundary, then continue the equality through
the noncontact annulus, and finally use the sphere-valued rigidity theorem in
the contact ball.

\begin{proposition}\label{prop:outer-annulus}
Every minimizer $u$ agrees with $u_{a,n}$ on
\begin{equation}\label{eq:outer-agreement}
 \B_1\setminus\overline{\B_R}.
\end{equation}
\end{proposition}

\begin{proof}
We first recover the normal derivative on the outer boundary.  By
\cref{thm:propagation}, $m_n=M_n$ as functions on $(0,1)$, so $m_n$ is
smooth.
The derivative formula \eqref{eq:mprime-C} applies to every minimizer at every $a$, and \eqref{eq:Mprime} gives
\[
 C=\frac a2M_n'(a)=C_{\rm rad}.
\]
Also $E=M_n$. The Pohozaev identity for $u$ and for the radial map gives
\[
 P=\sigma_{n-1}\beta_R^2.
\]
By \eqref{eq:Eouter},
\[
 E-C=M_n-C_{\rm rad}=\sigma_{n-1}\beta_R.
\]
Thus equality holds in the Cauchy--Schwarz estimate
\[
 \int_{\partial\B_1}x\cdot\partial_\nu u
 \le\|x\|_{L^2(\partial\B_1)}
 \|\partial_\nu u\|_{L^2(\partial\B_1)}.
\]
Equality in Cauchy--Schwarz in the real Hilbert space
$L^2(\partial\B_1;\R^n)$ yields
$\partial_\nu u=\lambda x$ for a constant scalar $\lambda$ almost everywhere
on the boundary.  Since
\[
 \|x\|_{L^2(\partial\B_1)}^2=\sigma_{n-1},
 \qquad
 \|\partial_\nu u\|_{L^2(\partial\B_1)}^2
 =P=\sigma_{n-1}\beta_R^2,
\]
we have $|\lambda|=\beta_R$.  The boundary pairing is
$\lambda\sigma_{n-1}=E-C=\sigma_{n-1}\beta_R$, so
$\lambda=\beta_R>0$. Smoothness in the collar upgrades this equality to the
pointwise identity
\begin{equation}\label{eq:normal-data}
 \partial_\nu u=\beta_Rx
 \qquad\text{on }\partial\B_1.
\end{equation}

Both $u$ and $u_{a,n}$ are harmonic in a fixed collar of $\partial\B_1$.
They have the same Dirichlet data on $\partial\B_1$, and
\eqref{eq:normal-data} gives the
same Neumann data.  Put $v=u-u_{a,n}$ and shrink the collar, if necessary,
so that $v$ is smooth and harmonic on
\[
 \cC_\delta^-=\{x:1-\delta<|x|<1\}.
\]
Thus, componentwise,
\[
 v=0,
 \qquad
 \partial_\nu v=0
 \quad\text{on }\partial\B_1.
\]
We now spell out the continuation across the outer boundary.  On the
two-sided annulus
\[
 \cC_\delta=\{x:1-\delta<|x|<1+\delta\},
\]
define the zero extension
\[
 \widetilde v(x)=
 \begin{cases}
 v(x),&1-\delta<|x|<1,\\
 0,&1\le|x|<1+\delta.
 \end{cases}
\]
For any scalar component and any
$\phi\in C_c^\infty(\cC_\delta)$, Green's identity on
$\cC_\delta^-$ gives
\begin{align*}
 \langle\Delta\widetilde v,\phi\rangle
 &=\int_{\cC_\delta^-}v\,\Delta\phi\dd x\\
 &=\int_{\cC_\delta^-}(\Delta v)\phi\dd x
   +\int_{\partial\B_1}
   \bigl(v\,\partial_\nu\phi-\phi\,\partial_\nu v\bigr)\dd S.
\end{align*}
There is no term on $\{|x|=1-\delta\}$ because the support of $\phi$ stays
away from that boundary. Moreover, $\Delta v=0$ in $\cC_\delta^-$ and the two
Cauchy data in the preceding display give
\[
 \int_{\partial\B_1}
 \bigl(v\,\partial_\nu\phi-\phi\,\partial_\nu v\bigr)\dd S=0.
\]
Consequently,
\[
 \langle\Delta\widetilde v,\phi\rangle=0
 \qquad\text{for every }\phi\in C_c^\infty(\cC_\delta),
\]
so $\Delta\widetilde v=0$ in $\cD'(\cC_\delta)$. Weyl's lemma makes
$\widetilde v$ a smooth harmonic function on $\cC_\delta$. Since it vanishes
on the nonempty open set $\{1<|x|<1+\delta\}$, unique continuation on the
connected annulus gives
\[
 \widetilde v\equiv0\quad\text{in }\cC_\delta.
\]
Consequently, $u=u_{a,n}$ in an inner boundary collar.

Let
\[
 A=\{x:R<|x|<1\},
 \qquad
 \Omega_+=\{x\in\B_1:d_u(x)>0\}=\{|u|>a\},
\]
where $d_u$ is the continuous representative from
\cref{prop:local-consequences}. Let $G$ be the connected component of
$\Omega_+\cap A$ containing the boundary collar. By
\cref{prop:local-consequences}, every singular point has a neighborhood on
which $d_u=0$. Hence $\Omega_+\subset\operatorname{Reg}u$, and
\eqref{eq:EL} reduces to $\Delta u=0$ on $\Omega_+$. Thus both maps are
harmonic on $G$. Each component of $u-u_{a,n}$ is therefore real analytic
on the connected set $G$ and vanishes on a nonempty open subset. The identity
theorem for real-analytic functions gives
\begin{equation}\label{eq:agree-G}
 u=u_{a,n}\quad\text{on }G.
\end{equation}

We show that $G$ is relatively closed in $A$. Let $x_j\in G$ and $x_j\to x\in A$. By \eqref{eq:agree-G} and continuity of $d_u$,
\begin{align*}
 d_u(x)
 &=\lim_{j\to\infty}d_u(x_j)\\
 &=\lim_{j\to\infty}\bigl(w_R(|x_j|)-a\bigr)
 =w_R(|x|)-a>0,
\end{align*}
because $|x|>R$ and $w_R$ is strictly increasing there. Thus
$x\in\Omega_+\cap A$. Every connected component is closed relative to the
ambient set, so the fact that $G$ is a component of $\Omega_+\cap A$ implies
$x\in G$. Hence $G$ is relatively closed in $A$. On the other hand,
$\Omega_+\cap A$ is open in $A$, and connected components of open subsets of
$\R^n$ are open; therefore $G$ is also open in $A$. Since $A$ is connected
and $G$ is nonempty, $G=A$.
\end{proof}

We can now complete the proof of the main theorem by treating the contact
ball.

\begin{proof}[Proof of \cref{thm:main}]
Existence follows from \cref{prop:finite-existence}. Let $u$ be any minimizer.
By \cref{prop:outer-annulus}, the following exterior equality holds.
\[
 u=u_{a,n}\quad\text{on }\B_1\setminus\overline{\B_R}
\]
The restrictions of a global $W^{1,2}$ map to the two sides of the Lipschitz interface $\partial\B_R$ have the same Sobolev trace. Hence the exterior equality implies that the interior trace on $\partial\B_R$ is
\begin{equation}\label{eq:trace-R}
 u=a h,
 \qquad
 h(x)=\frac{x}{|x|}.
\end{equation}
Moreover, the outer energies agree. Since the total energies agree by \cref{thm:propagation},
\begin{equation}\label{eq:inner-energy-equality}
 E(u;\B_R)=E(u_{a,n};\B_R)
 =a^2E(h;\B_R).
\end{equation}

Inside $\B_R$, put
\[
 \rho=|u|\ge a,
 \qquad
 \omega=\frac{u}{|u|}\in W^{1,2}(\B_R;\Sph^{n-1}).
\]
The trace relation \eqref{eq:trace-R} gives $\omega=h$ on $\partial\B_R$. The polar decomposition \eqref{eq:polar-decomp} yields
\begin{align}
 E(u;\B_R)
 &=\int_{\B_R}\bigl(|D\rho|^2+\rho^2|D\omega|^2\bigr)\dd x\notag\\
 &\ge a^2\int_{\B_R}|D\omega|^2\dd x.\label{eq:inner-lower1}
\end{align}
After rescaling $\B_R$ to $\B_1$, we use the sharp identity-boundary
theorem for sphere-valued maps.  In dimension three this is
\cite[Theorem~7.1]{BCL1986}; in dimensions four through six it is
\cite[Theorem~9]{Hong2001} with $p=2$, whose hypothesis
$1<p\le n-1$ is satisfied.  These results give
\begin{equation}\label{eq:Hong-R}
 \int_{\B_R}|D\omega|^2\dd x
 \ge\int_{\B_R}|Dh|^2\dd x,
\end{equation}
with equality only if $\omega=h$ almost everywhere. Combining
\eqref{eq:inner-energy-equality}, the polar decomposition, and
\eqref{eq:Hong-R}, we obtain the following equality of nonnegative
quantities:
\[
 \begin{aligned}
 0
 &=E(u;\B_R)-a^2E(h;\B_R)\\
 &=\int_{\B_R}
 \bigl(|D\rho|^2+(\rho^2-a^2)|D\omega|^2\bigr)\dd x
 +a^2\bigl(E(\omega;\B_R)-E(h;\B_R)\bigr).
 \end{aligned}
\]
The first integral and the last energy difference are nonnegative because
$\rho\ge a$ and \eqref{eq:Hong-R} holds.  Hence every term vanishes.  In
particular, $D\rho=0$ almost everywhere. Its boundary trace is $a$, so
$\rho\equiv a$, and equality holds in \eqref{eq:Hong-R}.  The uniqueness
clause in the cited identity-boundary theorem gives $\omega=h$. Thus
\[
 u=ah=u_{a,n}\quad\text{in }\B_R.
\]
Together with the outer-annulus equality, this proves uniqueness on all of $\B_1$.
\end{proof}

\section*{Acknowledgments}
The authors are supported by National Key R\&D Program of China
2025YFA1017603. 
The authors acknowledge the use of AI tools. 
All mathematical statements and proofs were independently verified by the authors, who take full responsibility for the content of the manuscript.

\begingroup
\small
\bibliographystyle{amsplain}
\bibliography{radial_minimality}

@article{Baldes1984,
  author  = {Baldes, A.},
  title   = {Stability and uniqueness properties of the equator map from a ball into an ellipsoid},
  journal = {Math. Z.},
  volume  = {185},
  number  = {4},
  year    = {1984},
  pages   = {505--516},
  doi     = {10.1007/BF01236259}
}

@article{BCL1986,
  author  = {Brezis, H. and Coron, J.-M. and Lieb, E. H.},
  title   = {Harmonic maps with defects},
  journal = {Comm. Math. Phys.},
  volume  = {107},
  year    = {1986},
  pages   = {649--705},
  doi     = {10.1007/BF01205490}
}

@article{BernandMantelMuratovSimon2021,
  author  = {Bernand-Mantel, A. and Muratov, C. B. and Simon, T. M.},
  title   = {A quantitative description of skyrmions in ultrathin ferromagnetic
             films and rigidity of degree {$\pm1$} harmonic maps from
             {$\mathbb R^2$} to {$\mathbb S^2$}},
  journal = {Arch. Ration. Mech. Anal.},
  volume  = {239},
  year    = {2021},
  pages   = {219--299},
  doi     = {10.1007/s00205-020-01575-7}
}

@misc{DengIgnatLamy2025,
  author        = {Deng, B. and Ignat, R. and Lamy, X.},
  title         = {The conformal limit for bimerons in easy-plane chiral magnets},
  year          = {2025},
  eprint        = {2506.11955},
  archivePrefix = {arXiv},
  primaryClass  = {math.AP},
  note          = {Accepted for publication in SIAM J. Math. Anal.; arXiv:2506.11955}
}

@article{DengSunWei2024,
  author  = {Deng, B. and Sun, L. and Wei, J.-C.},
  title   = {Quantitative stability of harmonic maps from {$\mathbb R^2$} to
             {$\mathbb S^2$} with a higher degree},
  journal = {Calc. Var. Partial Differential Equations},
  volume  = {63},
  year    = {2024},
  pages   = {Paper No. 101, 34 pp.},
  doi     = {10.1007/s00526-024-02712-w}
}

@article{Duzaar1987,
  author  = {Duzaar, F.},
  title   = {Variational inequalities and harmonic mappings},
  journal = {J. Reine Angew. Math.},
  volume  = {374},
  year    = {1987},
  pages   = {39--60},
  doi     = {10.1515/crll.1987.374.39}
}

@article{DuzaarFuchs1986,
  author  = {Duzaar, F. and Fuchs, M.},
  title   = {Optimal regularity theorems for variational problems with obstacles},
  journal = {Manuscripta Math.},
  volume  = {56},
  year    = {1986},
  pages   = {209--234},
  doi     = {10.1007/BF01172157}
}

@misc{FGKSsing,
  author       = {Figalli, A. and Guerra, A. and Kim, S. and Shahgholian, H.},
  title        = {Constraint maps: singularities vs free boundaries},
  year         = {2024},
  eprint       = {2407.21128},
  archivePrefix = {arXiv},
  note         = {Preprint}
}

@article{FKSobstacle,
  author  = {Figalli, A. and Kim, S. and Shahgholian, H.},
  title   = {Constraint maps with free boundaries: the obstacle case},
  journal = {Arch. Ration. Mech. Anal.},
  volume  = {248},
  year    = {2024},
  pages   = {Article 79},
  doi     = {10.1007/s00205-024-02032-5}
}

@article{FGKSreview,
  author  = {Figalli, A. and Guerra, A. and Kim, S. and Shahgholian, H.},
  title   = {Constraint maps: insights and related themes},
  journal = {La Matematica},
  volume  = {5},
  year    = {2026},
  pages   = {Article 26},
  doi     = {10.1007/s44007-026-00209-w}
}

@article{Hong2001,
  author  = {Hong, M.-C.},
  title   = {On the minimality of the {$p$}-harmonic map {$x/|x|:\B^n\to\Sph^{n-1}$}},
  journal = {Calc. Var. Partial Differential Equations},
  volume  = {13},
  year    = {2001},
  pages   = {459--468},
  doi     = {10.1007/s005260100082}
}

@article{GuerraLamyZemas2025,
  author  = {Guerra, A. and Lamy, X. and Zemas, K.},
  title   = {Sharp quantitative stability of the {M\"obius} group among
             sphere-valued maps in arbitrary dimension},
  journal = {Trans. Amer. Math. Soc.},
  volume  = {378},
  year    = {2025},
  pages   = {1235--1259},
  doi     = {10.1090/tran/9272}
}

@article{LinWang2006,
  author  = {Lin, F.-H. and Wang, C.-Y.},
  title   = {Stable stationary harmonic maps to spheres},
  journal = {Acta Math. Sin. (Engl. Ser.)},
  volume  = {22},
  year    = {2006},
  pages   = {319--330},
  doi     = {10.1007/s10114-005-0673-7}
}

@article{Ramanathan1986,
  author  = {Ramanathan, J.},
  title   = {A remark on the energy of harmonic maps between spheres},
  journal = {Rocky Mountain J. Math.},
  volume  = {16},
  year    = {1986},
  pages   = {783--790},
  doi     = {10.1216/RMJ-1986-16-4-783}
}

@misc{Rupflin2023,
  author        = {Rupflin, M.},
  title         = {Sharp quantitative rigidity results for maps from
                   {$\mathbb S^2$} to {$\mathbb S^2$} of general degree},
  year          = {2023},
  eprint        = {2305.17045},
  archivePrefix = {arXiv},
  primaryClass  = {math.AP},
  note          = {Preprint, arXiv:2305.17045}
}

@article{SchoenUhlenbeck1982,
  author  = {Schoen, R. and Uhlenbeck, K.},
  title   = {A regularity theory for harmonic maps},
  journal = {J. Differential Geom.},
  volume  = {17},
  number  = {2},
  year    = {1982},
  pages   = {307--335},
  doi     = {10.4310/jdg/1214436923},
  note    = {Correction: J. Differential Geom. 18 (1983), 329}
}

@article{Topping2023,
  author  = {Topping, P. M.},
  title   = {A rigidity estimate for maps from {$\mathbb S^2$} to
             {$\mathbb S^2$} via the harmonic map flow},
  journal = {Bull. Lond. Math. Soc.},
  volume  = {55},
  number  = {1},
  year    = {2023},
  pages   = {338--343},
  doi     = {10.1112/blms.12731}
}

@article{Caffarelli1998,
  author  = {Caffarelli, L. A.},
  title   = {The obstacle problem revisited},
  journal = {J. Fourier Anal. Appl.},
  volume  = {4},
  number  = {4--5},
  year    = {1998},
  pages   = {383--402},
  doi     = {10.1007/BF02498216}
}

@article{ChenMusina1990,
  author  = {Chen, Y. M. and Musina, R.},
  title   = {Harmonic mappings into manifolds with boundary},
  journal = {Ann. Scuola Norm. Sup. Pisa Cl. Sci. (4)},
  volume  = {17},
  number  = {3},
  year    = {1990},
  pages   = {365--392}
}

@article{CoronGulliver1989,
  author  = {Coron, J.-M. and Gulliver, R. D.},
  title   = {Minimizing {$p$}-harmonic maps into spheres},
  journal = {J. Reine Angew. Math.},
  volume  = {401},
  year    = {1989},
  pages   = {82--100},
  doi     = {10.1515/crll.1989.401.82}
}

@misc{FGKSBernoulli,
  author        = {Figalli, A. and Guerra, A. and Kim, S. and Shahgholian, H.},
  title         = {Constraint maps with free boundaries: the {Bernoulli} case},
  year          = {2023},
  eprint        = {2311.03006},
  archivePrefix = {arXiv},
  primaryClass  = {math.AP},
  note          = {To appear in J. Eur. Math. Soc.; arXiv:2311.03006}
}

@article{FGKSNotices2025,
  author  = {Figalli, A. and Guerra, A. and Kim, S. and Shahgholian, H.},
  title   = {Constraint maps and free boundaries},
  journal = {Notices Amer. Math. Soc.},
  volume  = {72},
  number  = {5},
  year    = {2025},
  pages   = {494--503},
  doi     = {10.1090/noti3162}
}

@article{HardtKinderlehrerLin1988,
  author  = {Hardt, R. and Kinderlehrer, D. and Lin, F.-H.},
  title   = {Stable defects of minimizers of constrained variational principles},
  journal = {Ann. Inst. H. Poincar\'e Anal. Non Lin\'eaire},
  volume  = {5},
  number  = {4},
  year    = {1988},
  pages   = {297--322},
  doi     = {10.1016/S0294-1449(16)30340-7}
}

@article{HardtLin1987,
  author  = {Hardt, R. and Lin, F.-H.},
  title   = {Mappings minimizing the {$L^p$} norm of the gradient},
  journal = {Comm. Pure Appl. Math.},
  volume  = {40},
  number  = {5},
  year    = {1987},
  pages   = {555--588},
  doi     = {10.1002/cpa.3160400503}
}

@article{HardtLin1989,
  author  = {Hardt, R. and Lin, F.-H.},
  title   = {Stability of singularities of minimizing harmonic maps},
  journal = {J. Differential Geom.},
  volume  = {29},
  number  = {1},
  year    = {1989},
  pages   = {113--123},
  doi     = {10.4310/jdg/1214442637}
}

@article{Hildebrandt1972,
  author  = {Hildebrandt, S.},
  title   = {On the regularity of solutions of two-dimensional variational
             problems with obstructions},
  journal = {Comm. Pure Appl. Math.},
  volume  = {25},
  number  = {4},
  year    = {1972},
  pages   = {479--496},
  doi     = {10.1002/cpa.3160250407}
}

@article{Hong2000,
  author  = {Hong, M.-C.},
  title   = {On the {J\"ager--Kaul} theorem concerning harmonic maps},
  journal = {Ann. Inst. H. Poincar\'e Anal. Non Lin\'eaire},
  volume  = {17},
  number  = {1},
  year    = {2000},
  pages   = {35--46},
  doi     = {10.1016/S0294-1449(99)00103-1}
}

@article{IwaniecKovalevOnninen2011,
  author  = {Iwaniec, T. and Kovalev, L. V. and Onninen, J.},
  title   = {The {Nitsche} conjecture},
  journal = {J. Amer. Math. Soc.},
  volume  = {24},
  number  = {2},
  year    = {2011},
  pages   = {345--373},
  doi     = {10.1090/S0894-0347-2010-00685-6}
}

@article{IwaniecOnninen2012,
  author  = {Iwaniec, T. and Onninen, J.},
  title   = {{$n$}-Harmonic Mappings between Annuli: The Art of Integrating
             Free Lagrangians},
  journal = {Mem. Amer. Math. Soc.},
  volume  = {218},
  number  = {1023},
  year    = {2012},
  pages   = {viii+105}
}

@article{JagerKaul1979,
  author  = {J\"ager, W. and Kaul, H.},
  title   = {Uniqueness and stability of harmonic maps and their {Jacobi} fields},
  journal = {Manuscripta Math.},
  volume  = {28},
  number  = {1--3},
  year    = {1979},
  pages   = {269--291},
  doi     = {10.1007/BF01647975}
}

@article{JagerKaul1983,
  author  = {J\"ager, W. and Kaul, H.},
  title   = {Rotationally symmetric harmonic maps from a ball into a sphere
             and the regularity problem for weak solutions of elliptic systems},
  journal = {J. Reine Angew. Math.},
  volume  = {343},
  year    = {1983},
  pages   = {146--161},
  doi     = {10.1515/crll.1983.343.146}
}

@article{Lin1987,
  author  = {Lin, F.-H.},
  title   = {A remark on the map {$x/|x|$}},
  journal = {C. R. Acad. Sci. Paris S\'er. I Math.},
  volume  = {305},
  number  = {12},
  year    = {1987},
  pages   = {529--531}
}

@incollection{Lin2016FreeBoundary,
  author    = {Lin, F.-H.},
  title     = {Lectures on Elliptic Free Boundary Problems},
  booktitle = {Lectures on the Analysis of Nonlinear Partial Differential
               Equations, Part 4},
  editor    = {Chemin, J.-Y. and Lin, F.-H. and Zhang, P.},
  series    = {Morningside Lectures in Mathematics},
  volume    = {4},
  publisher = {Higher Education Press and International Press},
  address   = {Beijing and Boston},
  year      = {2016},
  pages     = {115--193}
}

@article{Luckhaus1988,
  author  = {Luckhaus, S.},
  title   = {Partial {H\"older} continuity for minima of certain energies
             among maps into a {Riemannian} manifold},
  journal = {Indiana Univ. Math. J.},
  volume  = {37},
  number  = {2},
  year    = {1988},
  pages   = {349--367},
  doi     = {10.1512/iumj.1988.37.37017}
}

@article{SchoenUhlenbeck1984,
  author  = {Schoen, R. and Uhlenbeck, K.},
  title   = {Regularity of minimizing harmonic maps into the sphere},
  journal = {Invent. Math.},
  volume  = {78},
  number  = {1},
  year    = {1984},
  pages   = {89--100},
  doi     = {10.1007/BF01388715}
}

@article{Tomi1972,
  author  = {Tomi, F.},
  title   = {Variationsprobleme vom {Dirichlet}-{Typ} mit einer {Ungleichung}
             als {Nebenbedingung}},
  journal = {Math. Z.},
  volume  = {128},
  year    = {1972},
  pages   = {43--74},
  doi     = {10.1007/BF01111513}
}
\endgroup

\bigskip
\noindent
(Bin Deng) School of Mathematics and Statistics, Wuhan University, Wuhan,
Hubei Province, P.R. China, 430072.\\
Email address: \href{mailto:dbmath@whu.edu.cn}{dbmath@whu.edu.cn}
\par\smallskip
\noindent
(Jiahuan Li) School of Mathematical Sciences, University of Science and
Technology of China, Hefei, 230026, Anhui Province, P.R. China.\\
Email address:
\href{mailto:jiahuan@mail.ustc.edu.cn}{jiahuan@mail.ustc.edu.cn}
\par\smallskip
\noindent
(Yilu Liu) School of Mathematical Sciences, University of Science and
Technology of China, Hefei, 230026, Anhui Province, P.R. China.\\
Email address:
\href{mailto:liuylgeoanaly@mail.ustc.edu.cn}{liuylgeoanaly@mail.ustc.edu.cn}
\par\smallskip
\noindent
(Xi-Nan Ma) School of Mathematical Sciences, University of Science and
Technology of China, Hefei, 230026, Anhui Province, P.R. China.\\
Email address: \href{mailto:xinan@ustc.edu.cn}{xinan@ustc.edu.cn}

\end{document}